\documentclass[11pt,reqno]{amsart}
\usepackage{amsfonts,amsthm,amssymb}
\usepackage{amsmath}
\usepackage{pdfsync}
\usepackage{verbatim}
\usepackage[mathscr]{eucal}
\usepackage{times}
\usepackage{graphicx}
\allowdisplaybreaks

\newtheorem{thm}{Theorem}[section]

\newtheorem{lem}[thm]{Lemma}
\newtheorem{prop}{Proposition}
\newtheorem{cor}[thm]{Corollary}

\newtheorem{definition}{Definition}[section]
\newtheorem{remark}{Remark}

\numberwithin{equation}{section}

\newcommand{\Q}{\ensuremath{\mathbb{Q}}}
\newcommand{\R}{\ensuremath{\mathbb{R}}}
\newcommand{\Z}{\ensuremath{\mathbb{Z}}}

\newcommand{\N}{\ensuremath{\mathbb{N}}}
\newcommand{\PP}{\ensuremath{\mathbb{P}}}

\newcommand{\E}{\ensuremath{\mathbb{E}}}

\newcommand{\G}{\ensuremath{\mathcal{G}}}

\newcommand{\1}{\ensuremath{\mathbf{1}}}
\newcommand{\bv}{\ensuremath{\underline{v}}}
\newcommand{\bx}{\ensuremath{\underline{x}}}

\newcommand{\bnu}{\ensuremath{\underline{\nu}}}

\newcommand{\X}{\ensuremath{\underline{X}}}
\newcommand{\Y}{\ensuremath{\underline{Y}}}

\newcommand{\bX}{\ensuremath{\underline{X}}}
\newcommand{\bXi}{\ensuremath{\underline{X}_I}}

\newcommand{\de}{\ensuremath{\delta}}

\newcommand{\subs}{\ensuremath{\subseteq}}

\newcommand{\eq}{\begin{equation}
\newcommand{\ee}{\end{equation}}}

{ \theoremstyle{claim} }
{ \theoremstyle{notation} }
{ \theoremstyle{convention} }

\newcommand{\aLL}{\ensuremath{\overline{\Lambda}}}

\newcommand{\x}{\underline{x}}
\newcommand{\y}{\underline{y}}
\newcommand{\w}{\underline{\omega}}
\newcommand{\0}{\underline{0}}

\numberwithin{equation}{subsection}

\begin{document}

\title{Corners in dense subsets of $\mathbf{P}^d$}
\author{\'Akos Magyar and Tatchai Titichetrakun}
\email{tatchai@math.ubc.ca}

\email{magyar@math.ubc.ca}
\maketitle
\begin{abstract}
Let $\PP^d$ be the $d$-fold direct product of the set of primes. We prove that if $A$ is a subset of $\PP^d$ of positive relative upper density then $A$ contains infinitely many ``corners", that is sets of the form\\ $\{x,x+te_1,...,x+te_d\}$ where $x \in \Z^d, t \in \Z$ and $\{e_1,..,e_d\}$ are the standard basis vectors of $\Z^d$. The main tools are the hypergraph removal lemma, the linear forms conditions of Green-Tao and the transference principles of Gowers and Reingold et al.
\end{abstract}

\section{Introduction} \numberwithin{equation}{section} A remarkable result in additive number theory due to Green and Tao \cite{GT1} proves the existence of arbitrary long arithmetic progressions in the primes. It roughly states that if $A$ is a subset of the primes of positive relative upper density then $A$ contains arbitrary constellations, that is non-trivial affine copies of any finite set of integers. It is closely related to Szemer\'edi's theorem \cite{SZ} on the existence of long arithmetic progressions in dense subsets of the integers, in fact it might be viewed as a relative version of that.
Another basic result in this area is the multi-dimensional extension of Szemer\'edi's theorem first proved by Furstenberg and Katznelson \cite{FK1}, which states that if $A\subs\Z^d$ is of positive upper density then $A$ contains non-trivial affine copies of any finite set $F\subs\Z^d$. The proof in \cite{FK1} uses ergodic methods however a more recent combinatorial approach was developed by Gowers \cite{GW1} and also independently by Nagel, R\"odl and Schacht \cite{R1}.
\\\\
It is natural to ask if both results have a common extension, that is if the Furstenberg-Katznelson theorem can be extended to subsets of $\PP^d$ of positive relative upper density, that is when the base set of integers are replaced by that of the primes. In fact, this question was raised by Tao \cite{TA2}, where the existence of arbitrary constellations among the Gaussian primes was shown. A partial result, extending the original approach of \cite{GT1}, was obtained earlier by B. Cook and the first author \cite{MC1}, where it was proved that relative dense subsets of $\PP^d$ contain an affine copy of any finite set $F\subs\Z^d$ which is in \emph{general position}, meaning that each coordinate hyperplane contains at most one point of $F$.
\\\\
However when the set $F$ is not in general position, it does not seem feasible to find a suitably pseudo-random measure supported essentially on the $d$-tuples of the primes, due to the self-correlations inherent in the direct product structure. For example, if we want to count corners $\{(a,b),(a+d,b),(a,b+d)\}$ in $A\subs \PP^2$ then if $(a+d,b),(a,b+d) \in \PP^2$ then the remaining vertex $(a,b)$ must also be in $\PP^2$. Thus the probability that all three vertices are in $\PP^2$ (or in the direct product of the almost primes) is not $(\log\,N)^{-6}$ as one would expect, but roughly $(\log\,N)^{-4}$, preventing the use of any measure of the form $\nu\otimes\nu$.
\\\\
In light of this our method is different, based on the hypergraph approach partly used already in \cite{TA2}, where one reduces the problem to that of proving a hypergraph removal lemma for weighted uniform hypergraphs. The natural approach is to use an appropriate form of the so-called transference principle \cite{GW2}, \cite{RTTV} to remove the weights and apply the removal lemmas for ``un-weighted" hypergraphs, obtained in \cite{GW1}, \cite{R1}, \cite{TA3}. This way our argument also covers the main result of \cite{MC1} and in particular that of \cite{GT1}. Very recently another proof of the (one dimensional) Green-Tao theorem and the main result of \cite{TA2}, based on a removal lemma for uniform hypergraphs, has been given in \cite{CFZ}. An interesting feature of the argument there is that it only uses the so-called linear forms condition of \cite{GT1}.
\\

Recall that a set $A \subseteq \PP^d$ has upper relative density $\alpha$ if
\[
\limsup_{N \rightarrow \infty}\frac{|A \cap \PP_N^d |}{|\PP_{N}^d|} = \alpha
\]
Let us state our main result.
\begin{thm}
\label{MainTheorem}
Let $A \subset (\mathbb{P}_N)^d $ with positive relative upper density $\alpha>0$
then $A$ contains at least $C(\alpha)\frac{N^{d+1}}{(\log N)^{2d}}$
corners for some (computable) constant $C(\alpha)>0$.
\end{thm}

As mentioned above, we will use the hypergraph approach which has been used to establish the the existence of corners (and then that of general constellations) in dense subsets $\Z^d$ \cite{GW1} \cite{R1}. This was first observed, in the the case of 2-dimensional corners by Solymosi \cite{S1}, where the key tool was to apply to so-called triangle removal lemma of Ruzsa and Szemer\'edi. In our weighted setting, this method allows us to distribute the weights so that we can avoid dealing with higher moments of the Green-Tao measure $\nu$. We will define the notion and prove some facts for independent weight systems for which the weight systems related to corners is just a special case. The reason that we cannot handle arbitary constellations is that we don't quite have a suitable removal lemma (e.g. Thm. 5.1) for general weight systems on non-uniform hypergraphs. Indeed for general constellations our approach leads to a weighted hypergraph with weights possibly attached to any lower dimensional hyperedge, making it difficult to apply transference principles to remove the weights. Thus one needs different ideas which is addressed by Cook and the authors in a separate paper \cite{CoMaTi}. Simultaneously a completely different approach, based on an "infinite form" of the linear forms condition and on the recent work on inverse Gowers conjectures \cite{GT2}, \cite{GT3}, \cite{GT4}, has also been developed by Tao and Ziegler \cite{TZ}.

\subsection{Notation}
$[N]:=\{ 1,2, ..., N\}, [M,N]:=\{M,M+1,...,M+N\}, \mathbb{P}_N := \mathbb{P} \cap [N].$\\
Write $\x=(x_1,...,x_d), \y=(y_1,...,y_d),\w \in \{0,1\}^d,$ let $P_{\w}:\Z_N^{2d} \rightarrow \Z_N^d$ be the projection defined by
\[
P_{\w}(\x,\y) = \underline{u}=(u_1,...,u_d), u_j= \begin{cases} x_j &\text{if}\quad \omega_j=0 \\  y_j &\text{if}\quad \omega_j=1  \end{cases}
\]
For each $I \subseteq [d], \x_I=(x_i)_{i \in I}.$ We may denote $\x$ for $\x_{[d]}$ when we work in $\Z_N^d.$
$\w_I$ means elements in $\{0,1\}^{|I|}$. Similarly we may write $\w$ for $\w_{[d]}$. We also define $P_{\w_I}(\x_I,\y_I)$ in the same way. $\w |_{I}$ is the $\w$ restricted to the index set $I.$ \\
For finite sets $X_j, j\in [d], I \subseteq [d]$ then $\X_I :=\prod_{j \in I}X_j$ and
\[
P_{\w_I}(\X_I,\Y_I)=\prod_{i \in I}Z_i,~~~~~~~~~~~~~~~~~~~~~~~~~~~~~~~~~ Z_i= \begin{cases} X_i &, \w_I(i)=0 \\  Y_i &, \w_I(i)=1
\end{cases}\]
If we want to fix on some position, we can write  for example $ \w_{(0,[2,d])}$ means element in $\{0,1\}^d$ such that the first position is $0$.\\
Also  for each $\w$, define $\y_{1(\w)} \in \Z^d$ by
\[
(\y_{1(\w)})_i =
\begin{cases}
0 &\text{if}\quad \omega_i=0 \\  y_i &\text{if}\quad \omega_i=1
\end{cases}
, 1 \leq i \leq d.\]
$\y_{0(\w)} \in \Z^d$ is also defined similarly. \\
For any finite set $X$ and $f:X \rightarrow \R$, and for any measure $\mu$ on $X$,
\[
\E_{x \in X}f(x):=\frac{1}{|X|}\sum_{x \in X}f(x), \int_X f d\mu:=\frac{1}{|X|}\sum_{x \in X}f(x)\mu(x)
\]
Unless otherwise specified, the error term $o(1)$ means a quantity that goes to $0$ as $N,W \rightarrow \infty.$

\section{Weighted hypergraphs and box norms.} \numberwithin{equation}{section}

\subsection{Hypergraph setting.}
First let us parameterize any affine copies of a corner as follows
\begin{definition} A non-degenerate corner is given by the following set of $d-$tuples of size $d+1$ in $\Z^d$ (or $\Z_N^d$):
\[
\{ (x_1,...,x_d),(x_1+s,x_2,...,x_d),...,(x_1,...,x_{d-1},x_d+s),s \neq 0 \}
\]
or equivalently,
\[
\{(x_1,...,x_d),(z-\sum_{\substack{1 \leq j \leq d\\ j \neq 1}}x_j,x_2,...,x_d),(x_1,z-\sum_{\substack{1 \leq j \leq d\\ j \neq 2}}x_j , x_3,...,x_d),...,(x_1,...,x_{d-1},z-\sum_{\substack{1 \leq j \leq d\\ j \neq d}}x_j )   \}
\]
with $z \neq \sum_{1 \leq i \leq d}x_i$
\end{definition}

Now to a given set $A \subseteq \Z_N^d$, we assign a $(d+1)-$ partite hypergraph $\mathcal{G}_A$ as follows: \par
Let ($X_1,...,X_{d+1}$): $X_1=...=X_{d+1} :=\Z_N $ be the vertex sets, and for $j \in [1,d]$ let an element $a \in X_j$ represent the hyperplane $x_j=a$, and an element $a \in X_{d+1}$ represent the hyperplane $a=x_1+..+x_d$. We join these $d$ vertices (which represent $d$ hyperplanes) if all of these $d$ hyperplanes intersect in $A$. Then a simplex in $\G_A$ corresponds to a corner in $A$. Note that this includes trivial corners which consist of a single point.
\par  For each $I \subseteq [d+1]$ let $E(I)$ denote the set of hyperedges whose elements are exactly from vertices set $V_i,\ i \in I$. In order to count corners in $A$, we will place some weights on some of these hyperedges that will represent the coordinates of the corner. To be more precise we define the weights on $1-$edges:
\[
\nu_j(a)=\nu(a),a \in X_j, j \leq d, \hspace{3mm} \nu_{d+1}(a)=1,\ a \in X_{d+1},
\]
and on $d-$hyperedges:
 \[
 \nu_I(a)=\nu(a_{d+1}-\sum_{j \in I \backslash \{d+1\}}a_j) ,\ \ a \in E(I),\ |I|=d,\ d+1 \in I
 \]
 \[
 \nu_{[1,d]}(a)=1 ,\ a \in E([1,d])
 \]

In particular the weights are $1$ or of the form $\nu_I(L_I(\x_I))$ where all linear forms $\{L_I(\x_I)\}$ are pairwise linearly independent. This is an example of something we call independent weight system (see definition 2.1). Note that we can also parameterize any configuration of the form $\{\x,\x+t\bv_1,\ldots,x+t\bv_d\}$ in $\PP^d$ using an appropriate independent weight system.
Now for each $I=[d+1]\backslash \{j\}, 1 \leq j \leq d$ let
 $$f^I=\1_A(x_1,...,x_{j-1},x_{d+1}-\sum_{\substack{1 \leq i \leq d \\ i \neq j}}x_i,x_{j+1},...,x_d)\cdot \nu_I $$
 and for $I=[d]$ let $f^{I}=\1_A(x_1,...,x_d).$ As the coordinates of a corner contained in $\PP^d$ are given by  $2d$ prime numbers, we define
 \begin{align*} \label{Lambda}
 \Lambda :=\Lambda_{d+1}(f^{I},|I|=d):=\E_{\x_{[d+1]}}\prod_{|I|=d}f^I\prod_{i=1}^d\nu(x_i)
         &= \sum_{\substack{p_i \in A, 1 \leq i \leq 2d \\(p_i)_{1 \leq i \leq 2d}~~~~~~~~~ \text{constitutes a corner}}}\prod_{i=1}^{2d}\nu(p_i) \\
         &\approx \frac{\log^{2d}N}{N^{d+1}}|\{\text{number of corners in $A$} \}|
 \end{align*}
 Hence $\Lambda$ can be used to estimate the numbers of corners (ignoring W-trick here and assuming that $\nu(N) \approx \log N$ for now). Indeed if $\Lambda \geq C_1$ then
 \[
 |\{\text{number of corners in $A$} \}| \geq C_2 \frac{N^{d+1}}{\log^{2d}N}.
 \]

We define measure spaces associated to our system of measure as follows.  For $1 \leq i \leq d$, let $(X_i,d\mu_{X_i})=(\Z_N,\nu)$ where $\nu$ is the Green-Tao measure, and let $\mu_{X_{d+1}}$ be the normalized counting measure on $X_{d+1}=\Z_N$. With this notation one may write
\[
\Lambda:=\Lambda_{d+1}(f^{I},|I|=d)=\int_{X_1}\cdots \int_{X_{d+1}}\prod_{|I|=d}f^I\ d\mu_{X_1} \cdots d\mu_{X_{d+1}}.
\]
This is indeed a special case of
 \begin{definition}[Independent weight system]
  \label{weight}
An independent weight system is a family of weights on the edges of a $d+1-$partite hypergraph such that for any $I \subseteq [d+1],|I| \leq d$, $\nu_I(\x_I)$ is either $1$ or of the form $\prod_{j=1}^{K(I)}\nu(L^j_I(\bx_I))$ where all distinct linear forms $\{L_I^j \}_{I \subseteq [d+1],\ 1 \leq j \leq K(I)}$ are pairwise linearly independent, moreover the form $L_I^j$ depends exactly on the variables $\x_I=(x_j)_{j\in I}$.
\end{definition}

In fact for a weight system that arised from parametrizing affine copies of configurations in $\Z^d$, it is easy to see from the construction that for any $I \subseteq [d+1],|I|=d$ all distinct linear forms $\{ L^k_J\}_{J \subseteq I,1 \leq k \leq K(J)}$ are linearly independent however we don't need this fact in our paper.
We define a measure on $\X_I,I \subseteq [d+1],|I|=d$ associated to an independent weight system by
\[
\int_{\X_{[d]}}f d\mu_{\X_{[d]}}:=\E_{\x_{[d]}}f^I \cdot \prod_{I \subseteq [d],|I|< d}\nu_I(\x_I),
\]
as well as on $\X_{[d+1]}$  by
\[
\int_{\X_{[d+1]}}f d\mu_{\X_{[d+1]}}:=\E_{\x_{[d+1]}}f \cdot \prod_{I \subseteq [d+1],|I|< d}\nu_I(\x_I),
\]
and the associated multi-linear form by
 \eq \label{Lambda}
 \Lambda(f^I,|I|=d):= \int_{\X_{[d+1]}}\prod_{|I|=d}f^I d\mu_{\X_{[d+1]}}
 \ee

\subsection{Basic Properties of Weighted Box Norm}
In this section we describe the weighted version of Gowers's uniformity (box) norms and the so-called Gowers's inner product associated to the hypergraph $\mathcal{G}_A$ endowed with a weight system $\{\nu_I\}_{I\subs\,[d+1],|I|\leq d}$.
\begin{definition}
For each $1 \leq j \leq d,$ let $X_j,Y_j$ be finite set (in this paper we will define $X_j=Y_j:=\Z_N$) with a weight system $\bnu$ on $\X_{[d]} \times \Y_{[d]}$. For $f:\X_{[d]} \rightarrow \mathbb{R} $, define
\begin{align*}
\left\|f\right\|_{\Box_{\nu}}^{2^d} &:= \int_{\X_{[d]} \times \Y_{[d]} } \prod_{\w_{[d]}}f(P_{\w_{[d]}}(\x_{[d]},\y_{[d]}))d\mu_{\X_{[d]} \times \Y_{[d]}} \\
&:= \E_{\x_{[d]}}\E_{\y_{[d]}}\prod_{\w_{[d]}}f(P_{\w_{[d]}}(\x_{[d]},\y_{[d]}))\prod_{|I|<d}\prod_{\w_I}\nu_I(P_{\w_I}(\x_I,\y_I))
\end{align*}
and define the corresponding Gowers's inner product of $2^d$ functions,
\begin{align*}
\left\langle f_{\w},\w \in \{0,1\}^d \right\rangle_{\Box_{\nu}} &:= \int_{\X_{[d]} \times \Y_{[d]}} \prod_{\w_{[d]}}f_{\w_{[d]}}(P_{\w_{[d]}}(\x_{[d]},\y_{[d]}))d\mu_{\X_{[d]} \times \Y_{[d]}} \\
&:= \E_{\x_{[d]}}\E_{\y_{[d]}}\prod_{\w_{[d]}}f_{\w_{[d]}}(P_{\w_{[d]}}(\x_{[d]},\y_{[d]}))\prod_{|I|<d}\prod_{\w_I}\nu_I(P_{\w_I}(\x_I,\y_I))
\end{align*}
So $\left\langle f ,\w \in \{0,1\}^d \right\rangle_{\Box_{\nu}}=\left\|f\right\|_{\Box_{\nu}}^{2^d}.$
\end{definition}
\begin{definition}[Dual Function] For $f,g: \Z_N^d \rightarrow \R$ define the weight inner product
\[
\left\langle f,g \right\rangle_{\nu}:=\int_{X_{[d]}}f \cdot g ~~~~~~~~~~~~~~~~~~~~~~~~~~~~~~~ d\mu_{\X_{[d]}} =\E_{\x \in \Z_N^d}f(\x)g(\x)\prod_{|I|<d}\nu_I(\x_I).
\]
Define the dual function of $f$ by
\[
\mathcal{D}f :=\E_{\y \in \Z_N^d}\prod_{\w \neq \underline{0}}f(P_{\w}(\x,\y))\prod_{|I|<d}\prod_{\w_I \neq \underline{0}}\nu_I(P_{\w_I}(\x_I,\y_I))
\]
So
\begin{align*}
\left\|f\right\|_{\Box_{\nu}}^{2^d}&= \E_{\x \in \Z_N^d}f(\x)\prod_{|I|<d}\nu_I(\x_I)\bigg[ \E_{\y \in \Z_N^d}\prod_{\w \neq \underline{0}}f(P_{\w}(\x,\y))\prod_{|I|<d}\prod_{\w_I \neq \underline{0}}\nu_I(P_{\w_I}(\x_I,\y_I)) \bigg] \\
&=\left\langle f, \mathcal{D}f \right\rangle_{\nu}
\end{align*}
\end{definition}
It may not be clear immediately from the definition that  $\left\| \cdot \right\|_{{\Box^d_{\nu}}}$ is a norm but this will follow from the following theorem whose statements and the strategies of the proof are similar to analogue theorem for ordinary Gowers inner product.
\begin{thm}[Gowers-Cauchy-Schwartz's Inequality] $|\left\langle f_{\w}; \w \in \{0,1\}^d \right\rangle| \leq  \displaystyle{ \prod_{\w_{[d]}} } \left\|f_{\w}\right\|_{\Box^d_{\nu}}.$
\end{thm}
\begin{proof} We will use Cauchy-Schwartz's inequality and linear form condition. Write
\begin{align*}
\left\langle f_{\omega} ; \omega \in \{0,1\}^d \right\rangle_{\Box_{\nu}^d}
&=\E_{\x_{[2,d]},\y_{[2,d]}}\bigg[ \bigg( \prod_{|I|<d,1 \notin I} \prod_{\w_I} \nu_I(P_{\w_I}(\x_I,\y_I))   \bigg)^{1/2} \\
& \bigg( \E_{x_1} \nu(x_1) \prod_{\w_{[2,d]}} f_{\w_{(0,[2,d])}}(x_1,P_{\w_{[2,d]}}(\x_{[2,d]},\y_{[2,d]}))\prod_{|I|<d-1,1 \notin I}\nu_{\{1\} \cup I}(x_1,P_{\w_I}(\x_I,\y_I))\bigg) \\
&\times \bigg( \prod_{|I|<d,1 \notin I} \prod_{\w_I} \nu_I(P_{\w_I}(\x_I,\y_I))   \bigg)^{1/2} \\
& \bigg( \E_{y_1} \nu(y_1) \prod_{\w_{[2,d]}} f_{\w_{(1,[2,d])}}(y_1,P_{\w_{[2,d]}}(\x_{[2,d]},\y_{[2,d]}))\prod_{|I|<d-1,1 \notin I}\nu_{\{1\} \cup I}(y_1,P_{\w_I}(\x_I,\y_I))\bigg)\bigg]
\end{align*}
Applying the Cauchy Schwartz inequality in the $\x_{[2,d]},\y_{[2,d]}$ variables, one has
\[
|\langle f_{\w} ; \w \in \{0,1\}^d \rangle_{\Box_{\nu}^d} |^2 \leq A \cdot B
\]
here,
\begin{align*}
A &=\E_{\x_{[2,d]},\y_{[2,d]}} \bigg[\prod_{|I|<d,1 \notin I} \prod_{\w_I} \nu_I(P_{\w_I}(\x_I,\y_I))   \ \\
&\times \bigg( \E_{x_1,y_1}\nu(x_1)\nu(y_1) \prod_{\w_{[2,d]}} f_{\w_{(0,[2,d])}}(x_1,P_{\w_{[2,d]}}(\x_{[2,d]},\y_{[2,d]}))f_{\w_{(0,[2,d])}}(y_1,P_{\w_{[2,d]}}(\x_{[2,d]},\y_{[2,d]}))  \\
&\times \prod_{|I|<d-1,1 \notin I}\prod_{\w_I}\nu_{\{1\} \cup I}(x_1,P_{\w_I}(\x_I,\y_I))\nu_{\{1\} \cup I}(y_1,P_{\w_I}(\x_I,\y_I))\bigg)  \bigg] \\
&=\left\langle f^{(0)}_{\w}(P_{\w}(\x_{[d]},\y_{[d]})) \right\rangle_{\Box^d_{\nu}}
\end{align*}
where
$
f^{(0)}_{\tilde{\w}}=f_{(0, \tilde{\w} \cap [2,d])}
$
for any $\tilde{\w}_{[1,d]}$. And,
\begin{align*}
B &=\E_{\x_{[2,d]},\y_{[2,d]}} \bigg[\prod_{|I|<d,1 \notin I} \prod_{\w_I} \nu_I(P_{\w_I}(\x_I,\y_I))   \ \\
&\times \bigg( \E_{x_1,y_1}\nu(x_1)\nu(y_1) \prod_{\w_{[2,d]}} f_{\w_{(1,[2,d])}}(x_1,P_{\w_{[2,d]}}(\x_{[2,d]},\y_{[2,d]}))f_{\w_{(1,[2,d])}}(y_1,P_{\w_{[2,d]}}(\x_{[2,d]},\y_{[2,d]}))  \\
&\times \prod_{|I|<d,1 \notin I}\prod_{\w_I}\nu_{\{1\} \cup I}(x_1,P_{\w_I}(\x_I,\y_I))\nu_{\{1\} \cup I}(y_1,P_{\w_I}(\x_I,\y_I))\bigg)  \bigg] \\
&=\left\langle f^{(1)}_{\w}(P_{\w}(\x_{[d]},\y_{[d]})) \right\rangle_{\Box^d_{\nu}}
\end{align*}
where
$
f^{(1)}_{\tilde{\w}}=f_{(1, \tilde{\w} \cap [2,d])}
$for any $\tilde{\w}_{[1,d]}.$ \\
In the same way, we apply Cauchy-Schwartz's inequality in $(\x_{[3,d]},\y_{[3,d]})$ variables to end up with
\[
|\langle  f_{\w}; \w \in \{0,1\}^d \rangle_{\Box^d_{\nu}}|^4 \leq
\prod_{\w_{[0,1]}}\langle  f^{\w_{[1,2]}}_{\w}; \w \in \{0,1\}^d \rangle_{\Box^d_{\nu}}
\]
In the same way, applying Cauchy-Schwartz's inequality consecutively in $(\x_{[4,d]},\y_{[4,d]}),...,(\x_{[d,d]},\y_{[d,d]})$ variables, we end up with
\begin{align*}
|\langle  f_{\w}; \w \in \{0,1\}^d \rangle_{\Box^d_{\nu}}|^{2^d} &\leq
\prod_{\w_{[d]}}\left\langle f^{\w},...,f^{\w} \right\rangle_{\Box^d_{\nu}}, f^{\w}=f_{\w} \\
&\leq \prod_{\w_{[d]}} \left\| f_{\w} \right\|^{2^d}_{\Box^d_{\nu}}
\end{align*}
\end{proof}
\begin{cor}
$\left\| \cdot \right\|_{{\Box^d_{\nu}}}$ is a norm for $N$ is sufficiently large.
\end{cor}
\begin{proof}
First we show nonnegativity. By the linear forms condition,
$
\|1\|_{\Box_{\nu}}=1+o(1).
$
Hence by the Gowers-Cauchy-Schwartz inequality, we have
$
\| f \|_{\Box_{\nu}^{d}} \gtrsim |\langle  f,1,...,1 \rangle_{\Box_{\nu}^{d}}| \geq 0
$
for all sufficiently large $N$.
 Now
\begin{align*}
\left\| f+g \right\|_{\Box_{\nu}^d} &= \left\langle  f+g,...,f+g \right\rangle_{\Box_{\nu}^d} \\
&=\sum_{\w \in \{0,1\}^d}\left\langle h^{\omega_1},...,h^{\omega_d} \right\rangle_{\Box_{\nu}^d}, h^{\omega}=\begin{cases} f  &,\omega=0 \\ g &,\omega=1 \end{cases}\\
&\leq \sum_{\w \in \{0,1\}^d} \left\| h^{\omega_1} \right\|_{\Box_{\nu}^d}...\left\| h^{\omega_d} \right\|_{\Box_{\nu}^d} =(\left\| f \right\|_{\Box_{\nu}^d}+\left\| g \right\|_{\Box_{\nu}^d})^{2^d}
\end{align*}
Also it follows directly from the definition that $\left\|\lambda f\right\|_{\Box_{\nu}^d}^{2^d}=\lambda^{2^d}\left\| f \right\|_{\Box_{\nu}^d}^{2^d}$. Since the norm are nonnegative, we have $\left\| \lambda f \right\|_{\Box_{\nu}^d}=|\lambda|\left\|f \right\|_{\Box_{\nu}^d}.$
\end{proof}

\subsection{Weighted generalized von-Neumann inequality}
The generalized von-Neumann inequality says that the average $\Lambda:=\Lambda_{d+1,\nu}(f^I,I \subseteq [d+1],|I|=d)$(see \eqref{Lambda}) is controlled by the weighted box norm.  We show this inequality in the general settings of an independent weight system.
\begin{thm}[Weighted generalized von-Neumann inequality]
Let $I \subseteq [d+1], |I|=d, f^I: \X_I \rightarrow \R.$ Let $\underline{\nu}$ be an independent system of measure on $\X_{[d+1]}$ that satisfies linear form conditions. Suppose $f^I$ are dominated by $\underline{\nu}$ i.e. $|f^I| \leq \nu_I$ then
\[
|\Lambda_{d+1,\nu}(f^{(1)},...,f^{(d+1)})| \lesssim \min\{\|f^{(1)}\|_{\Box_{\nu}^d},...,\|f^{(d+1)}\|_{\Box_{\nu}^d}\}
\]
\end{thm}
\begin{proof}
We will only use Cauchy-Schwartz inequality and the linear forms condition. The idea is to consider one of the variables say $x_j$, as a dummy variable and write
\[
\Lambda:=\E_{x_j}(...)\E_{\x_{[d+1]\backslash \{j\}}}(...)
\]
then apply Cauchy Schwartz's inequality to eliminate the lower complexity factors and use linear forms condition to control the extra factor gained. We do this repeatedly $d$ times.
\par First apply Cauchy-Schwartz's inequality in $x_{d+1}$ variable to eliminate $f^{(d+1)}$
\begin{align*}
|\Lambda| &\leq \E_{\x_{[d]}}f^{(d+1)}(\x_{[d]})\prod_{|I|<d,d+1 \notin I}\nu_I(\x_I)\bigg|\E_{x_{d+1}}\prod_{i=1}^d f^{(i)}(\x_{[d+1]\backslash \{i\}})\prod_{|I|< d,d+1 \in I}\nu_I(\x_I) \bigg| \\
&\leq \E_{\x_{[d]}}\bigg(\nu_{[d]}(\x_{[d]})\prod_{|I|<d,d+1 \notin I}\nu_I(\x_I) \bigg)^{1/2}\bigg(\nu_{[d]}(\x_{[d]})\prod_{|I|< d,d+1 \notin I}\nu_I(\x_I) \bigg)^{1/2}
\times \bigg|\E_{x_{d+1}} \prod_{i=1}^d f^{(i)} \prod_{|I|<d,d+1 \in I}\nu_I(\x_I) \bigg|
\end{align*}
Now by 
the linear forms condition (as the linear forms defining an independent weight system  are pairwise linearly independent), we have
\[
\E_{\x_{[d]}}\nu_{[d]}(\x_{[d]}) \prod_{|I|< d,d+1 \notin I}\nu_I(\x_{I})=1+o(1),
\]
hence
\begin{align*}
|\Lambda|^2 &\lesssim \E_{\x_{[d]}} \nu_{[d]}(\x_{[d]})\prod_{|I|<d,d+1 \notin I} \nu_I(\x_I) \\
&\times \E_{x_{d+1},y_{d+1}}\prod_{i=1}^d \prod_{\omega_{d+1} \in \{0,1\}}f^{(i)}(\x_{[d] \backslash \{i\}},P_{\omega_{d+1}}(x_{d+1},y_{d+1})) \prod_{\substack{|I|<d \\ d+1 \in I}} \prod_{\omega_{d+1} \in \{0,1\}}\nu_{I}( \x_{I \backslash \{d+1\}},P_{\omega_{d+1}}(x_{d+1},y_{d+1}))
\end{align*}
Next we want to eliminate $f^{(d)}(\x_{[d+1] \backslash \{ d\}}) \leq \nu_{[d+1] \backslash \{ d\}}(\x_{[d+1] \backslash \{ d\}})$. Write
\begin{align*}
|\Lambda|^2 &\lesssim \E_{\x_{[d+1]\backslash \{d\}},y_{d+1}} \prod_{\omega_{d+1} \in \{0,1\}}\nu_{[d+1] \backslash \{ d\}}(\x_{[d-1]},P_{\omega_{d+1}}(x_{d+1},y_{d+1})) \prod_{|I|< d, d \notin I} \prod_{\w_{\{d+1\} \cap I}}\nu_I(P_{\w_{I \cap \{d+1\}}}(\x_I,\y_I)) \\
&\times \prod_{\substack{|I|<d \\ d,d+1 \notin I}}\nu_I(\x_I) \E_{x_d}\prod_{i=1}^{d-1} f^{(i)}(\x_{[d] \backslash \{i\}},P_{\omega_{d+1}}(x_{d+1},y_{d+1})) \prod_{|I|<d,d \in I} \prod_{\w_{I \cap \{d+1\}}} \nu_I(P_{\w_{I \cap \{d+1\}}}(\x_I,\y_I))\cdot \nu_{[d]}(\x_{[d]})
\end{align*}
Again, by the 
linear forms condition on the face $P_{\w_{d+1}}(\X_{[d+1]\backslash \{d\}},\Y_{[d+1]\backslash \{d\}})$,
\[
\E_{\x_{[d+1]\backslash \{d\}},y_{d+1}} \prod_{\omega_{d+1} \in \{0,1\}}\nu_{[d+1]\backslash \{d\}}(\x_{[d-1]},P_{\omega_{d+1}}(x_{d+1},y_{d+1})) \prod_{|I|<d, d \notin I} \prod_{\w_{\{d+1\} \cap I}}\nu_I(P_{\w_{I \cap \{d+1\}}}(\x_I,\y_I))
\prod_{\substack{|I|< d\\ d,d+1 \notin I}} \nu_I(\x_I)
\]
which is $O(1)$ and hence
\begin{align*}
|\Lambda|^4 \lesssim &\E_{\x_{[d-2]},x_d,y_d,x_{d+1},y_{d+1}} \prod_{\w_{[d,d+1]}} \nu_{[d,d+1]}(P_{\w_{[d,d+1]}}(\x_{[d+1]\backslash \{d-1 \}},\y_{[d+1]\backslash \{d-1 \}})) \\
&\times \prod_{|I| \leq d, d-1 \notin I} \prod_{\w_{[d,d+1]\cap I}} \nu_I (P_{\w_{[d,d+1] \cap I}}(\x_I,\y_I)) \\
&\times \E_{x_{d-1}}\prod_{i=1}^{d-2}\prod_{\w_{[d,d+1]}}f^{(i)}(P_{\w_{[d,d+1]}}(\x_I,\y_I))\prod_{\substack{|I|\leq d \\ d-1 \in I}}\prod_{\w_{[d,d+1] \cap I}}\nu_I(P_{\w_{I \cap [d,d+1]} }(\x_I,\y_I)) \\
&\times \prod_{\omega_d} \nu_{[d]}(P_{\omega_d}(\x_{[d]},\y_{[d]}))\prod_{\omega_{d+1}} \nu_{[d+1] \backslash \{ d\}}(P_{\omega_{d+1}}(\x_{[d+1] \backslash \{d\}},\y_{[d+1] \backslash \{d\}}))
\end{align*}
Continue using Cauchy-Schwartz inequality in $x_{d-1},...,x_2$ in a similar fashion for
\[
\E_{\x_{[d+1] \backslash \{r\}},\y_{[d+1] \backslash \{r\}}}\prod_{\w_{[r+1,d+1]}}\nu_{[r+1,d+1]}(P_{\w_{[r+1,d+1]}}(\x_{[d+1] \backslash \{ r\}},\y_{[d+1] \backslash \{ r\}})) \prod_{\substack{|I|\leq d \\ r \notin I}}\prod_{\w_{[r+1,d+1] \cap I}}\nu_I(P_{\omega_{[r+1,d+1] \cap I}}(\x_I,\y_I))
\]
(which is $O(1)$ by linear forms are on all faces $P_{\w_{[r+1,d+1]}}(\X_{[d+1]\backslash \{d\}},\Y_{[d+1]\backslash \{d\}})$.) eventually we obtain
\begin{align}
|\Lambda|^{2^d} &\lesssim \E_{\x_{[2,d+1]},\y_{[2,d+1]}}\prod_{\w_{[2,d+1]}}f^{(1)}(P_{\w_{[2,d+1]}}(\x_{[2,d+1]},\y_{[2,d+1]})) \prod_{|I|<d, 1 \notin I} \prod_{\w_{[2,d+1] \cap I}}\nu_I(P_{\w_{[2,d+1] \cap I}}(\x_I,\y_I)) \notag \\
&\times W(P_{\w_{[2,d+1]}}(\x_{[2,d+1]},\y_{[2,d+1]})) \label{Equation200}
\end{align}
where
\begin{align*}
W:=W(P_{\w_{[2,d+1]}}(\x_{[2,d+1]},\y_{[2,d+1]})) &:= \E_{x_1}\prod_{|I|<d}\prod_{\w_{[2,d+1] \cap I}}\nu_I(x_1,P_{\w_{[2,d+1] \cap I}}(\x_{I \backslash \{1\}},\y_{I \backslash \{1\}})) \\
&\times \prod_{k=2}^{d+1} \prod_{\w_{[2,d+1]\backslash \{k\}}}\nu_{[d+1] \backslash \{k\}}(P_{\w_{[2,d+1]}}(\x_{[d+1] \backslash \{k\}},\y_{[d+1] \backslash \{k\}}))
\end{align*}
Write the RHS of \eqref{Equation200}
$=\left\| f^{(1)} \right\|^{2^d}_{\Box^d_{\nu}}+E$

where
\[
|E| \leq \E_{\x_{[2,d+1]},\y_{[2,d+1]}} \prod_{\w_{[2,d+1]}}f^{(1)}(P_{\w_{[2,d+1]}}(\x_{[2,d+1]},\y_{[2,d+1]}))
\prod_{|I|<d,1 \notin I} \prod_{\w_I} \nu_I(P_{\w_I}(\x_I,\y_I)) \times |W-1|
\]
We wish to show that $E=o(1)$. Now,
\begin{align*}
|E|^2 &\leq \E_{\x_{[2,d+1]},\y_{[2,d+1]}} \prod_{\w_{[2,d+1]}}\nu_{[2,d+1]}(P_{\w_{[2,d+1]}}(\x_{[2,d+1]},\y_{[2,d+1]}))\prod_{|I|<d,1 \notin I}\prod_{\w_I}\nu_I(P_{\w_I}(\x_I,\y_I)) \\
&\times \E_{\x_{[2,d+1]},\y_{[2,d+1]}} \prod_{\w_{[2,d+1]}}\nu_{[2,d+1]}(P_{\w_{[2,d+1]}}(\x_{[2,d+1]},\y_{[2,d+1]}))\prod_{|I|<d,1 \notin I}\prod_{\w_I}\nu_I(P_{\w_I}(\x_I,\y_I))|W-1|^2
\end{align*}
The term on the first line is $O(1)$ by linear form condition on all the faces $P_{\w_{[2,d+1]}}(\X_{[2,d+1]},\Y_{[2,d+1]})$.
So we just need to show
\begin{equation}\label{Equation201}
\E_{\x_{[2,d+1]},\y_{[2,d+1]}} \prod_{\w_{[2,d+1]}}\nu_{[2,d+1]}(P_{\w_{[2,d+1]}}(\x_{[2,d+1]},\y_{[2,d+1]}))\prod_{|I|<d,1 \notin I}\prod_{\w_I}\nu_I(P_{\w_I}(\x_I,\y_I))W =1+o(1)
\end{equation}

\begin{equation}\label{Equation202}
\E_{\x_{[2,d+1]},\y_{[2,d+1]}} \prod_{\w_{[2,d+1]}}\nu_{[2,d+1]}(P_{\w_{[2,d+1]}}(\x_{[2,d+1]},\y_{[2,d+1]}))\prod_{|I|<d,1 \notin I}\prod_{\w_I}\nu_I(P_{\w_I}(\x_I,\y_I))W^2 =1+o(1)
\end{equation}

\eqref{Equation201} follows from linear form conditions on the faces $\X_{[d+1]} \times \Y_{[2,d+1]}$. \eqref{Equation202} follows from linear form conditions on the faces $\X_{[d+1]} \times \Y_{[d+1]}$ and we are done.

\end{proof}



\section{The dual function estimate.}\numberwithin{equation}{section}
In this section we prove
\begin{thm}
For any independent measure system and any fixed $J \subseteq [d+1],|J|=d$, let $F_1,...,F_K:\X_J \rightarrow \R,\ F_j(\x_J) \leq \nu_J(\x_J)$ be given functions. Then for each $1 \leq j \leq K$ we have that
\[
\big\| \prod_{j=1}^K \mathcal{D}F_j \big\|^*_{\Box^d_{\nu}}=O_K(1)
\]
\end{thm}

\begin{proof}
We will denote by $I$ the subsets of a fixed set $J \subseteq [d+1],\ |J|=d$. First, write
\[
\mathcal{D}F_j(\x)=\E_{\y^j \in \Z_N^d}\prod_{\w \neq \underline{0}}F_j(P_{\w}(\x,\y^j))\prod_{|I|<d}\prod_{\w_I \neq \underline{0}}\nu_I(P_{\w_I}(\x_I,\y_I^j))
\]
Now assume $\left\| f \right\|_{\Box^d_{\nu}} \leq 1$ then
\begin{align*}
\big\langle f,\prod_{j=1}^K \mathcal{D}F_j \big\rangle_{\nu} &= \E_{\x \in \Z_N^d} f(\x) \prod_{j=1}^K \mathcal{D}F_j(\x)\prod_{|I|<d}\nu_I(\x_I) \\
&=\E_{\x \in \Z_N^d}f(\x)\E_{\y^1,...\y^K \in \Z_N^d}\prod_{j=1}^K \bigg[ \prod_{\w \neq \underline{0}} F_j(P_{\w}(\x,\y^j))\prod_{|I|<d}\big[\prod_{\w_I \neq \underline{0}} \nu_I(P_{\w_I}(\x_I,\y_I^j)) \big]\nu_I(\x_I)\bigg]
\end{align*}
We will compare this to the box norm to exploit the fact that $\left\| f \right\|_{\Box^d_{\nu}} \leq 1$. To compare this to the Gowers's inner product, let us introduce the following change of variables:
\par
For a fixed $\y \in \Z_N^d$, write $\y^j \mapsto \y^j+\y, 1 \leq j \leq K$ then our expression takes the form
\begin{align*}
\big\langle f, \prod_{j=1}^K \mathcal{D}F_j \big\rangle_{\nu}
&=\E_{\y^1,...,\y^K}\E_{\x}f(\x)\prod_{j=1}^K \bigg[\prod_{\w  \neq \underline{0}}F_j(P_{\w}(\x,\y+\y^j))\prod_{|I|<d}\prod_{\w_I \neq \underline{0}} \bigg[ \nu_I(P_{\w_I}(\x_I,+\y_I^j+\y_I)) \bigg] \nu_I(\x_I)   \bigg] \\
\end{align*}
This is equal to
\[
\E_{\y^1,...,\y^K}\E_{\x,\y} f(\x)\prod_{j=1}^K \bigg[\prod_{\w  \neq \underline{0}}F_j(P_{\w}(\x,\y+\y^j))\prod_{|I|<d}\prod_{\w_I \neq \underline{0}} \bigg[ \nu_I(P_{\w_I}(\x_I,\y_I^j+\y_I)) \bigg] \nu_I(\x_I)   \bigg]
\]
We will define functions $G_{\w,Y}(\x): \Z_N^d \rightarrow \R,\w \in \{0,1\}^d$ such that
\[
\big\langle f,\prod_{j=1}^K\mathcal{D}F_j \big\rangle_{\nu}=\E_{\y^1,..,\y^K}\left\langle G_{\w,Y}; \w \in \{0,1\}^d \right\rangle_{\Box^d_{\nu}}
\]
Now let $G_{\underline{0}}(\x):=f(\x)$ and for each $\tilde{\w} \neq \underline{0}, Y=(\y^1,...\y^K) \in (\Z_N^d)^K$, define
\[
G_{\tilde{\w},Y}(\x):=\prod_{j=1}^K\bigg[F_j(\x+\y^j_{1(\tilde{\w})})\bigg]\bigg( \prod_{|I|<d}\nu_I((\x+\y^j_{1(\tilde{\w})})\big|_I) \bigg)^{\frac{1}{2^{d-|I|}}} \prod_{|I|<d}\nu_I(\x_I)^{-\frac{1}{2^{d-|I|}}}
\]
Hence for $\tilde{\w} \neq \underline{0}$
\[
G_{\tilde{\w},Y}(P_{\tilde{\w}}(\x,\y))=\prod_{j=1}^K \bigg[ F_j(P_{\tilde{\w}}(\x,\y+\y^j))\big( \prod_{|I|<d}\nu_I((P_{\tilde{\w}}(( \x,\y+\y^j)\big|_I \big)^{\frac{1}{2^{d-|I|}}}\bigg] \prod_{|I|<d}\nu_I(P_{\tilde{\w}}(\x,\y)\big|_I)^{-\frac{1}{2^{d-|I|}}}
\]
\begin{remark}
For each $I \subseteq [d]$ and fixed $\w_I$, the number of $\w_{[d]}$ such that $\w_{[d]}|_I=\w_I$ is $2^{d-|I|}$ and
\[
P_{\w}(\x,\y)|_I=P_{\w_I}(\x_I,\y_I) \Longleftrightarrow \w|_I=\w_I
\]
\end{remark}
So from the remark
\begin{align*}
&\left\langle G_{\w,Y}; \w \in \{0,1\}^d \right\rangle_{\Box_{\nu}^d}=\E_{\x,\y \in \Z_N^d}\prod_{\w_{[d]}}G_{\w,Y}(P_{\w}(\x,\y))\prod_{|I|<d}\prod_{\w_I}\nu_I(P_{\w_I}(\x_I,\y_I)) \\
&= \E_{\x,\y \in \Z_N^d} \prod_{\w} \bigg[\prod_{j=1}^K \bigg[ F_j(P_{\w}(\x,\y+\y^j))(\prod_{|I|<d}\nu_I(P_{\w}( (\x,\y)+\y_{1(\w)}^j)\big|_I )^{\frac{1}{2^{d-|I|}}} \bigg]  \\
&\times \prod_{|I|<d}\nu_I(P_{\w}(\x_I,\y_I)|_I)^{-\frac{1}{2^{d-|I|}}} \bigg] \times \prod_{|I|<d}\prod_{\w_I}\nu_I(P_{\w_I}(\x_I,\y_I)) \\
&=\E_{\x,\y \in \Z_N^d}f(\x)\prod_{j=1}^K \prod_{\w \neq \underline{0}}F_j(P_{\w}(\x,\y+\y^j)) \prod_{|I|<d}\bigg[\prod_{j=1}^K\prod_{\w_I \neq \underline{0}} \nu_I(P_{\w_I}(\x_I,\y_I^j+\y_I)) \bigg] \nu_I(\x_I)
\end{align*}Hence we have
\[
\langle f,\prod_{j=1}^K\mathcal{D}F_j \rangle_{\nu}=\E_{\y^1,..,\y^K}\left\langle G_{\w}; \w \in \{0,1\}^d \right\rangle_{\Box^d_{\nu}}
\]
Then by Gowers-Cauchy-Schwartz's inequality, we have
\[
\big|\langle f, \prod_{j=1}^K \mathcal{D}F_j \rangle_{\nu} \big| \leq \left\|f \right\|_{\Box_{\nu}^d}\prod_{\w \neq \underline{0}}\left\|G_{\w,Y} \right\|_{\Box_{\nu}^d} \lesssim 1 +\sum_{\w_{[d]} \neq \underline{0}}\left\| G_{\w,Y}\right\|^{2^d}_{\Box_{\nu}^d}
\]
Hence to prove the dual function estimate, it is enough to show that
\[
\E_{\y^1,...,\y^K}\left\| G_{\tilde{\w},Y} \right\|_{\Box_{\nu}^d}^{2^d}=O_K(1)
\]
For any fixed $\tilde{\w} \neq \underline{0}$. Now
\begin{align*}
\E_{\y^1,...,\y^K}\left\| G_{\tilde{\w},Y} \right\|_{\Box^d_{\nu}}^{2^d}
&=\E_{\y^1,...,\y^K} \E_{\x,\y}\prod_{\w}G_{\tilde{\w},Y}(P_{\w}(\x,\y))\prod_{|I|<d}\prod_{\w_I}\nu_I(P_{\w_I}(\x_I,\y_I))\\
&\leq\E_{\y^1,...,\y^K} \E_{\x,\y} \prod_{\w}\prod_{j=1}^K \bigg[ \nu_{[d]}(P_{\w}(\x,\y)+\y^j_{1(\tilde{\w})})\prod_{|I|<d}\nu_I((P_{\w}(\x,\y)+\y^j_{1(\tilde{\w})})\big|_I)^{\frac{1}{2^{d-|I|}}} \\
&\times \prod_{|I|<d}\nu_I(P_{\w}(\x,\y) \big|_I)^{-\frac{1}{2^{d-|I|}}}\bigg] \prod_{|I|<d}\prod_{\w_I}\nu_I(P_{\w_I}(\x_I,\y_I)) \\
&=\E_{\y^1,...,\y^K} \E_{\x,\y}\prod_{j=1}^K \bigg[\prod_{\w} \nu_{[d]}(P_{\w}((\x,\y)+\y^j_{1(\tilde{\w})}\big|_I)) \prod_{|I|<d}\prod_{\w_I}\nu_I(P_{\w_I}(\x_I,\y_I)+\y^j_{1(\tilde{\w})}\big|_I) \bigg]
\end{align*}
by remark 1 above. As the linear forms appearing in the above expression are pairwise linearly independent  this is $O_K(1)$ (in fact it is $O(1)$ if $N$ is sufficiently large w.r.t. to $K$) by the linear forms condition as required.
\end{proof}

\section{Transference Principle}\numberwithin{equation}{section}
In this section, we will slightly modify the transference principle in \cite{GW2}(see Theorem 4.6)  , which will allow us to deduce results for functions dominated by a pseudo-random measure from the corresponding result on bounded functions. We will do this on the set on which our functions have bounded dual, and treat the contributions of the remaining set as error terms. \par
We will work on functions $f: \X_I \rightarrow \R$, dominated by $\nu_I$. WLOG assume $I=[d]$. Let $\left\langle  \cdot \right\rangle$ be any inner product on $\mathcal{F}:= \{f:\X_{[d]} \rightarrow \R \}$ written as $\left\langle f,g  \right\rangle= \int f \cdot g ~~~~d \mu$ for some measure $\mu$ on $\X_{[d]}.$
 In this section we will need the explicit discription of the set $\Omega(T)$ that the dual function is bounded by $T$ using the correlation condition (see appendix).
\subsection{Dual Boundedness on $\X_I$}
One property of the dual functions that is used in \cite{GT1} is their boundedness. However in the weighted settings, this is generally not true. To get around this, we will be working on sets on which the dual functions are bounded and treat the contributions of the remaining parts as error terms. \par
Consider any independent weight system. Let $I \subseteq [d+1],|I|=d$, $f:\X_I \rightarrow \R, |f| \leq \nu_I$ (WLOG $I=[d] $). Recall
\begin{equation*}
\mathcal{D}(f) = \E_{\y} \prod_{\w \neq \0} \nu(L(P_{\w}(\x,\y)))\prod_{|I|<d}\prod_{\w_I \neq \0}\nu_I(P_{\w_I}(\x_I,\y_I))
\end{equation*}
Write $h_{\w_I}=L^I(\x)|_{0(\w_I)}$ hence using correlation condition (see appendix), we have
\begin{equation}
|\mathcal{D}(f)| \leq \prod_{\emptyset \neq J \subseteq [d]}\sum_{(\w_{I_1},\w_{I_2})\in T_J}\tau(W \cdot (a_{\w_{I_1}}h_{\w_{I_1}}-a_{\w_{I_2}}h_{\w_{I_2}})+(a_{\w_{I_1}}-a_{\w_{I_2}})b)
\end{equation}
\label{Equation401}
where for each $J \subsetneq [d],J \neq \emptyset$
\[
T_J :=\{\{\w_{I_1},\w_{I_2}\}, \w_{I_1},\w_{I_2} \neq \underline{0},\w_{I_1} \neq \w_{I_2},1(\w_{I_1})=1(\w_{I_2})=J :\exists c \in \Q, L^{I_1}(\y_{1(\w_{I_1})})=cL^{I_2}(\y_{1(\w_{I_2})})  \}
\]
where $a_{\w_{I_j}} \in \mathbb{Q}$ are some constants. Define
\begin{equation}
\Omega_J(T)=\{(\x_{[d]}: \sum_{\{\w_{I_1},\w_{I_2}\}\in T_J}\tau(W \cdot (a_{\w_{I_1}}h_{\w_{I_1}}-a_{\w_{I_2}}h_{\w_{I_2}})+(a_{\w_{I_1}}-a_{\w_{I_2}})b)) \leq T^{1/2^d}\}
\end{equation}
\label{Equation402}
\begin{equation}
\Omega(T)= \bigcap_{J \subsetneq [d]} \Omega_J(T)
\end{equation}
So $\mathcal{D}f$ is bounded by $T$ on $\Omega(T)$ for any fixed $T>1$.

\subsection{Transference principle}
\begin{definition}
For each $T>1$ we have the set $\Omega(T)$ and define the following sets
\begin{align*}
\mathcal{F} &:= \{f:\X_{[d]}  \rightarrow \R \} \\
\mathcal{F}_T &:= \{ f \in \mathcal{F} : \text{supp}(f) \subseteq \Omega(T) \} \\
\mathcal{S}_T &:= \{ f \in \mathcal{F}_T : |f| \leq \nu_{[d]}(\x_{[d]})+2 \}
\end{align*}
We define the following (basic anti-correlation) norm on $\mathcal{F}_T$
\[
\left\| f \right\|_{\text{BAC}} := \max_{g \in \mathcal{S}_T }|\left\langle f,\mathcal{D}g \right\rangle|
\]
\end{definition}
We have the following basic properties of this norm.
\begin{prop}$\ $

\begin{enumerate}
    \item $g \in \mathcal{F}_T \Rightarrow \mathcal{D}g \in \mathcal{F}_T$

	\item $\left\| \cdot \right\|_{\text{BAC}}$ is a norm on $\mathcal{F}_T$ and can be extended to be a seminorm on $\mathcal{F}$. Furthermore, we have $\left\| f \right\|_{\text{BAC}}=\left\| f\cdot \1_{\Omega(T)} \right\|_{\text{BAC}}, f \in \mathcal{F}.$

	\item $\text{Span} \{ \mathcal{D}g : g \in \mathcal{S}_T \} =\mathcal{F}_T$

	\item $\left\| f \right\|^*_{\text{BAC}} = \inf\{\sum_{i=1}^k|\lambda_i|,f=\sum_{i=1}^k \lambda_i \mathcal{D}g_i; g_i \in \mathcal{S}_T \} $ for $f \in \mathcal{F}_T$
	\end{enumerate}
	\end{prop}
\begin{remark}
If $f \notin \mathcal{F}_T$ then $\text{supp}(f) \nsubseteq \Omega(T)$ so $f$ is not of the form $\sum_{i=1}^k \lambda_i \mathcal{D}g_i; g_i \in \mathcal{F}_T$ as RHS is zero.
\end{remark}	
\begin{proof}
\begin{enumerate}
	\item Suppose $(\tilde{x}_1,...,\tilde{x}_{d}) \in \Omega(T)^C$ then there is an $J \subsetneq [d]$ such that $K(\tilde{\x}_{[d] \backslash J})>T$ where $K$ is the function in the definition of $\Omega_J(T)$ for some $j$. Let $g \in \mathcal{F}_T$ then $g(\tilde{\x}_{[d] \backslash J},\x_J)=0$ for all $\x_J \in \underline{X}_J$ so $\mathcal{D}g \in \mathcal{F}_T$.
	\item It follows directly from the definition that $\left\| f+g \right\|_{\text{BAC}} \leq  \left\| f \right\|_{\text{BAC}}+ \left\| g \right\|_{\text{BAC}} $  and $\left\| \lambda f \right\|_{\text{BAC}}= |\lambda| \left\|   f \right\|_{\text{BAC}}$ for any $\lambda \in \R.$
Now suppose $f \in \mathcal{F}_T, f$ is not identically zero then we need to show that $\left\| f \right\|_{\text{BAC}} \neq 0.$ Since $X$ and $Z$ are finite sets, we have that $\left\|f \right\|_{\infty}=\max_{x,z}|f(x,z)|< \infty$. Let $g=\gamma f$ where $\gamma$ is a constant such that $\left\| g \right\|_{\infty}<2$ then $g \in \mathcal{S}_T$ and $\left\langle f, \mathcal{D}g \right\rangle =\left\langle f, \mathcal{D}\gamma f \right\rangle =\gamma^{2^d-1}\left\langle f,\mathcal{D}f \right\rangle >0$ so $\left\| f \right\|_{\text{BAC}}>0$ 	\\
Now $\text{supp}(\mathcal{D}g) \subseteq \Omega(T) $ we have for any $f \in \mathcal{F}$
\[
\left\| f \right\|_{\text{BAC}}=\sup_{g \in \mathcal{S}_T}|\left\langle f,\mathcal{D}g \right\rangle|=\sup_{g \in \mathcal{S}_T}|\left\langle f \cdot \1_{\Omega(T)},\mathcal{D}g \right\rangle|=\left\| f \cdot \1_{\Omega(T)} \right\|_{\text{BAC}}
\]
\item
If there is an $f \in \mathcal{F}_T, f$ is not identically zero and $f \notin \text{span}\{\mathcal{D}g:g \in \mathcal{S}_T \}$ So  $f \in \text{span}\{\mathcal{D}g:g \in \mathcal{S}_T \}^{\perp}$ then $\left\langle  f , \mathcal{D}g \right\rangle=0$ for all $g \in \mathcal{S}_T$. So $\left\| f \right\|_{BAC}=0 $ which is a contradiction.
\item
Define $\left\| f \right\|_D = \inf\{\sum_{i=1}^k |\lambda_i|:f=\sum_{i=1}^k \lambda_i \mathcal{D}g_i, g_i \in \mathcal{S}_T\}$ which can be easily verified to be a norm on $\mathcal{F}_T$. Now let $\phi,f \in \mathcal{F}_T , f=\sum_{i=1}^k \lambda_i \mathcal{D}g_i, g_i \in \mathcal{S}_T$, then
\[
|\left\langle  \phi ,f \right\rangle|=\sum_{i=1}^k |\lambda_i||\left\langle \phi, \mathcal{D}g_i \right\rangle| \leq \left\| \phi \right\|_{\text{BAC}}\sum_{i=1}^k |\lambda_i| \leq \left\| \phi \right\|_{\text{BAC}}\left\| f \right\|_D
\]
so
\[
\left\| f \right\|_{\text{BAC}}^* \leq \left\| f \right\|_D
\]
Next for all $g \in \mathcal{S}_T$, we have $\left\| \mathcal{D}g  \right\|_D \leq 1$  then
\[
\left\|f \right\|_{\text{BAC}}= \sup_{g \in \mathcal{S}_T}|\left\langle f, \mathcal{D}g\right\rangle | \leq \sup_{\left\| h \right\|_D \leq 1}|\left\langle f,h\right\rangle| = \left\| f \right\|_D^*
\]
so $\left\| f \right\|_{\text{BAC}} \leq \left\| f \right\|_D^*$ i.e. $\left\| f \right\|_{\text{BAC}}^* \geq \left\| f \right\|_D$.
So  $\left\| f \right\|_{\text{BAC}}^* = \left\| f \right\|_D$.
\end{enumerate}
\end{proof}
Now let us prove the following lemma whose proof relies on the dual function estimate. From here we consider our inner product $\left\langle \cdot \right\rangle_{\nu}$ and the norm $\left\| \cdot \right\|_{\Box_{\nu}}$. This argument also works for any norm for which one has the dual function estimate.
\begin{lem}
Let $\phi \in \mathcal{F}_T$ be such that $\left\| \phi \right\|^*_{\text{BAC}} \leq C$ and $\eta > 0$. Let $\phi_+ := \max\{0,\phi\}$. Then there is a polynomial $P(u)=a_mu^m+...+a_1u+a_0$ such that
\begin{enumerate}
	\item  $\left\| P(\phi)-\phi_+ \right\|_{\infty} \leq \eta$
	\item $\left\| P(\phi) \right\|_{\Box^d_{\nu}}^{*} \leq \rho(C,T,\eta)$ \\	
\end{enumerate}
where $$\rho(C,T,\eta):=2\inf R_P(C)$$
where the infimum is taken over polynomials $P$ such that $ \left\|P-\phi_+ \right\|_{\infty} \leq \eta ~~~~~~~\text{on} ~~~~~~~[-CT,CT]$ and
$$R_P(x)=\sum_{j=0}^m C(j)|a_j|x^j, ~~~~~~~\text{where} ~~~~~~~C(m) ~~~~~~~\text{is the constant in the dual function estimate}$$
\end{lem}
\begin{proof}
First, recall that if $(x_1,...,x_{d}) \in \text{supp}(\mathcal{D}g_i) \subseteq \Omega(T)$ then
\[
|\mathcal{D}g(x_1,...,x_{d})| \leq  T
\]
Now suppose $\left\| \phi \right\|_{\text{BAC}}^* \leq C$ then there exist $g_1,..,g_k \in \mathcal{S}_T$ and $\lambda_1,...,\lambda_k$ such that $\phi=\sum_{i=1}^k \lambda_i \mathcal{D}g_i$ and $\sum_{1 \leq i \leq k}|\lambda_i| \leq C$. Hence
\[
|\phi(x_1,...,x_{d})| \leq (\sum_{i=1}^k |\lambda_i|)(\max_{1 \leq i \leq k}|\mathcal{D}g_i(x_1,..,x_{d})|) \leq CT
\]
Hence \textit{the Range of} $\phi$ $=\phi(\Omega(T)) \subseteq [-CT,CT]$. Then by Weierstrass approximation theorem, there is a polynomial $P$ (which may depend on $C,T,\eta$) such that $R_P(C) \leq \rho$ and
\[
|P(u)-u_+| \leq \eta \quad \forall |u| \leq CT
\]
and so $\left\| P(\phi)-\phi_+ \right\|_{\infty} \leq \eta$ and we have (1). \\
Now using the dual function estimate, we have
\begin{align*}
\left\| \phi^m \right\|^*_{\Box^d_{\nu}} &\leq \big\| (\sum_{1 \leq i \leq k}\lambda_i \mathcal{D}g_i)^m\big\|^*_{\Box^d_{\nu}}
\leq \sum_{1 \leq i_1 \leq... \leq i_m \leq k}|\lambda_{i_1}...\lambda_{i_m}|\left\| \mathcal{D}g_{i_1}...\mathcal{D}g_{i_m} \right\|_{\Box^d_{\nu}}^* \\
&\leq C(m) \sum_{1 \leq i_1 \leq... \leq i_m \leq k} |\lambda_{i_1}...\lambda_{i_m}|
\leq C(m)(\sum_{1\leq i \leq k}|\lambda_i|)^m
\leq C(m)C^m
\end{align*}
Hence $\left\| P(\phi) \right\|_{\Box^d_{\nu}}^* \leq \sum_{m=0}^d |a_m|C(m)C^m \leq \rho(C,T,\eta)$
\end{proof}
Now we are ready to prove the transference principle.
\begin{thm}
\label{Equation497}
Suppose $\bnu$ is an independent weight system. Let $f \in \mathcal{F}$ and $0 \leq f(\x_{[d]}) \leq \nu_{[d]}(\x_{[d]})$, let $\eta>0$. Suppose $N \geq N(\eta,T)$ is large enough, then there are functions $g,h$ on $X_1 \times ... \times X_{d}$ such that
\begin{enumerate}
	\item $f=g+h$ on $\Omega(T)$
	\item $0 \leq g \leq2$ on  $\Omega(T)$
	\item $\left\| h \cdot \1_{\Omega(T)} \right\|_{\Box^d_{\nu}} \leq \eta$
\end{enumerate}
\end{thm}
To prove this theorem, it suffices to show
 \begin{thm}
 \label{Equation498}
  With the same assumption in Theorem 4.2, there are functions $g,h$ such that
\begin{enumerate}
	\item $f=g+h$ on $\Omega(T)$
	\item $0 \leq g \leq2$ on  $\Omega(T)$
	\item $\left\| h \cdot \1_{\Omega(T)} \right\|_{\text{BAC}} \leq \eta$
\end{enumerate}
Here the BAC-norm is the BAC-norm with respect to $\left\langle  \cdot \right\rangle_{\nu}$
\end{thm}
\begin{proof}[Theorem \ref{Equation498} $\Rightarrow$ Theorem \ref{Equation497}:] Since $h \cdot \1_{\Omega(T)}=f \cdot \1_{\Omega(T)}-g\cdot \1_{\Omega(T)}$, we have $-2 \leq h \cdot \1_{\Omega(T)} \leq \nu $ so
$|h \cdot \1_{\Omega(T)}| \leq \nu+2 $ so $h \cdot \1_{\Omega(T)} \in \mathcal{S}_T$. Hence by the definifion of BAC-norm,
\[
\eta \geq \| h \cdot \1_{\Omega(T)}\|_{\text{BAC}} \geq \langle h \cdot \1_{\Omega(T)}, \mathcal{D} (h \cdot \1_{\Omega(T)}) \rangle_{\nu} = \| h \cdot \1_{\Omega(T)}\|_{\Box^d_{\nu}}^{2^d}
\]
\end{proof}
The following lemma will be used in the next proof.
\begin{lem}[cororally 3.2 in \cite{GW2}]
Let $K_1,..K_r$ be closed convex subsets of $\R^d$, each containing 0 and suppose $f \in \mathbb{R}^d$ cannot be written as a sum $f_1+...+f_r, f_i \in c_iK_i,c_i>0.$ Then there is a linear functional $\phi$ such that $\left\langle f, \phi \right\rangle >1$ and $\left\langle g,\phi \right\rangle \leq c_i^{-1}$ for all $i \leq r$ and all $g \in K_i.$
\end{lem}
\begin{proof}[Proof of Theorem \ref{Equation498}:] Define
\begin{align*}
K &:= \{g \in \mathcal{F}: 0 \leq g \leq 2 ~~~~~~~~~\text{on}~~~~~~~~~\Omega(T) \} \\
L &:= \{h \in \mathcal{F}: \left\| h \right\|_{\text{BAC}} \leq \eta\}
\end{align*}
Then it is clear that $K,L$ are convex.(Also $0 \in K, 0 \in \text{Int}(L)$ and then $0 \in \text{Int}(K+L)$.)
Assume that $f \notin K+L$ on $\Omega(T)$ then by Lemma 4.4, there exists $\phi \in \mathcal{F}$ such that
\begin{enumerate}
	\item $\left\langle  \phi,f \cdot \1_{\Omega(T)} \right\rangle_{\nu} > 1$
	\item $\left\langle  \phi,g \right\rangle_{\nu} \leq 1 ~~~~~~~~~~~~~~~~~~~\forall g \in K$
	\item  $\left\langle  \phi,h \right\rangle_{\nu} \leq 1 ~~~~~~~~~~~~~~~~~~~\forall h \in L$
\end{enumerate}
First, we claim that $\phi \in \mathcal{F}_T.$ To see this, suppose $g$ is a function whose $\text{supp}(g)\subseteq \Omega(T)^C$ (i.e. $g \equiv 0$ on $\Omega(T)$ so $g \in K$.) Since $g \in K, \left\langle  \phi,g \right\rangle_{\nu} \leq 1$ but $g$ could be chosen arbitrarily on $\Omega(T)^C$ so we must have $\phi \big|_{\Omega(T)^C} \equiv 0$ and hence $\phi \in \mathcal{F}_T$. Now let
\[
g(\x_{[d]})= \begin{cases}
                        2 &\text{if}~~~~~~~~~~~~~~~~~~~~~~~~~~~~~~~~~~~~~~~~~~~~~~~~~~~~~~~~~~~~~~~~~~~~~~~~~~~~~~~~~~~~~~~~~~~~~~~~~~~~~~~~~~~~~~~~~~~ \phi(\x_{[d]})\geq 0 \\ 0 &\text{otherwise}
                        \end{cases}
\]
then $g \in K$ and
\[
\left\langle \phi,g \right\rangle_{\nu}=\left\langle \phi_+,2 \right\rangle_{\nu} =2\left\langle \phi_+,1 \right\rangle_{\nu}  \leq 1 \Rightarrow \left\langle \phi_+,1 \right\rangle_{\nu} \leq \frac 12
\]
Now since $\phi \in \mathcal{F}_T, h \in L$. Suppose $\| h \cdot \1_{\Omega(T)^C} \|_{\text{BAC}} \leq 1$ then we have
\[
\langle \phi, h \cdot \1_{\Omega(T)^C}\rangle_{\nu}=\langle \phi,h \rangle_{\nu} \leq \eta^{-1}.
\]
Hence if $h' \in \mathcal{F}_T$ and $\left\| h' \right\|_{\text{BAC}} \leq 1$ then $\left\| h' \cdot \1_{\Omega(T)} \right\|_{\text{BAC}} =\left\| h' \right\|_{\text{BAC}} \leq 1 $ so
\[
\langle  \phi,h' \rangle_{\nu} \leq \eta^{-1} \quad\forall h' \in \mathcal{F}_T, \left\| h' \right\|_{\text{BAC}} \leq 1
\]
so $\left\| \phi \right\|_{\text{BAC}}^* \leq \eta^{-1}$ as $\left\| \cdot \right\|_\text{BAC}$ is a norm on $\mathcal{F}_T$.\\

Now by the Lemma 4.1, there is a polynomial $P$ such that
\[
\left\| P(\phi)-\phi_+ \right\|_{\infty} \leq \frac 18
\]
and
\[
\left\| P(\phi) \right\|_{\Box^d_{\nu}}^* \leq \rho(C,T,\eta)
\]
Then $\left\langle P(\phi),1 \right\rangle_{\nu} \leq \left\langle P(\phi)-\phi_+,1 \right\rangle_{\nu}+ \left\langle \phi_+,1 \right\rangle_{\nu} \leq \frac 12 + \frac 18$
Also, from the definition of the weighted box norm and the linear form condition, we have
\[
\| \nu_{[d]}(\x_{[d]})-1 \|^{2^d}_{\Box^d_{\nu}}=o_{N \rightarrow \infty}(1)
\]
so suppose $N \geq N(T,\eta)$ then
\[
\left\langle P(\phi),\nu_{[d]}\right\rangle_{\nu}=\left\langle P(\phi),1\right\rangle_{\nu}+\left\langle P(\phi),\nu_{[d]}-1\right\rangle_{\nu}
\leq \frac 12+\frac 18 +\left\|P(\phi)\right\|_{\Box^d_{\nu}}^*\left\| \nu_{[d]}-1 \right\|_{\Box^d_{\nu}} \leq \frac 12 + \frac 14= \frac 34
\]
\[
|\left\langle \nu_{[d]}, \phi_+ \right\rangle_{\nu}| =|\left\langle \nu_{[d]}, \phi_+ -P(\phi)\right\rangle_{\nu}|+|\left\langle \nu_{[d]},  P(\phi)\right\rangle_{\nu}| \leq \left\| \phi_+-P(\phi) \right\|_{\infty}\left\langle \nu_{[d]},1 \right\rangle_{\nu}+\left\langle  \nu_{[d]},P(\phi)\right\rangle_{\nu} \leq \frac 18 \cdot \frac 12 +\frac 34
\]
Hence
\[
\left\langle  f \cdot \1_{\Omega(T)}, \phi \right\rangle_{\nu}=\left\langle  f, \phi \right\rangle_{\nu} \leq \left\langle  f,\phi_+ \right\rangle_{\nu} \leq \left\langle \nu_{[d]},\phi_+ \right\rangle_{\nu} \leq \frac 34 +\frac {1}{10}<1
\]
which is a contradiction. Hence $f \in K+L$ on $\Omega(T)$.
\end{proof}
Now we can rephrase Theorem \ref{Equation498} as follow:
\begin{thm}[Transference Principle]
\label{Equation499}
Suppose $\bnu$ is an independent weight system. Let $f \in \mathcal{F}, 0 \leq f \leq \nu$ and $0 < \eta <1 \ll T$ then there exists $f_1,f_2,f_3 \in \mathcal{F}$ such that
\begin{enumerate}
	\item $f=f_1+f_2+f_3 $
	\item $0 \leq f_1 \leq 2,~~~~~~~~~~~~~~ \text{supp}(f_1) \subseteq \Omega(T)$
	\item $\left\| f_2 \right\|_{\Box^d_{\nu}} \leq \eta, ~~~~~~~~~~~~~~ \text{supp}(f_2) \subseteq \Omega(T)$
	\item  $0 \leq f_3 \leq \nu,~~~~~~~~~~~~~~ \text{supp}(f_3) \subseteq \Omega(T)^C,\ \left\| f_3 \right\|_{L^1_{\nu}} \lesssim \frac 1T.$
\end{enumerate}
\end{thm}

\begin{proof}
Let $g,h$ be as in Theorem 4.3. Take $f_1=g \cdot \1_{\Omega_T},f_2=h \cdot \1_{\Omega_T}$ then $f \cdot \1_{\Omega_T}=f_1+f_2$. Let $f_3=f \cdot \1_{\Omega_T^C}$. Now by linear form condition

\begin{align*}
\left\| f_3 \right\|_{L^1_{\nu}} 
&\leq \frac{1}{T}\E_{\x_{[d]}} f \cdot \mathcal{D}f \cdot \prod_{I \subseteq [d],|I|<d}\nu_I(\x_I)   \\
&= \frac 1T\E_{\x_{[d]}}\E_{\y_{[d]}}\prod_{I \subseteq [d]}\nu_I(\x_I) \prod_{I \subseteq [d]}\prod_{\w_I \neq \underline{0}}\nu_I((P_{\w_I}(\x_I,\y_I)))
\lesssim \frac 1T
\end{align*}
\end{proof}

\section{Relative Hypergraph Removal Lemma}\numberwithin{equation}{section}
First let us recall the statement of ordinary functional hypergraph removal lemma \cite{TA3}.\footnote{In fact the paper \cite{TA3} proves this theorem only with the counting measure (with thenotion of $e-$discrepancy in place of Box norm). But the proof also works for any finite measure that has direct product structure (with the notion of weighted Box Norm).(see \cite{TA1} for the case of probability measures in $d=2,3$). However we don't know how to genralize this argument to arbitrary measure on the product space. If we can prove this theorem for any measure $\mu_{X_1 \times ... \times X_d}$ then we would be able to prove multidimensional Green-Tao's Theorem. } Recall the definition of $\Lambda$ in equation \eqref{Lambda}.
\begin{thm}
Given measure spaces $(X_1,\mu_{X_1}),...,(X_{d+1},\mu_{X_{d+1}})$ and $f^{(i)}: \bX_I  \rightarrow [0,1],I =[d+1] \backslash \{i\}$  Let $\epsilon >0,$
suppose $|\Lambda_{d+1}(f^{(1)},...,f^{(d)},f^{(d+1)})| \leq \epsilon.$
Then for $1 \leq i \leq d+1$, there exists $$E_i \subseteq X_{[d+1] \backslash \{i\}}$$
such that $\ \prod_{1 \leq j \leq d+1}\1_{E_j} \equiv 0 $  and for $1 \leq i \leq d+1,$\\
\[
\int_{X_1} \cdots \int_{X_{d+1}}  f^{(i)} \cdot \1_{E_i^C} d \mu_{X_1} \cdots d\mu_{X_d}d\mu_{X_{d+1}} \leq \delta(\epsilon)
\]
where $\delta(\epsilon) \rightarrow 0$ as $\epsilon \rightarrow 0.$
\end{thm}
Also let us state a functional version of Szemer\'edi's Regularity Lemma \cite{TA3} that we will use later in the proof. If $\mathcal{B}$ is a finite factor of $X$ i.e. a finite $\sigma-$algebra of measurable sets in $X$, then $\mathcal{B}$ is a partition of $X$ into atoms $A_1,...,A_M$. Let $f:X \rightarrow \R$ be measurable then we define the conditional expectation $\E(f| \mathcal{B}):X \rightarrow \R$ is defined by $\E(f| \mathcal{B})(x)=(1/|A_i|)\int_{A_i}f(x)d\mu_{X}$ if $x \in A_i$ (defined up to set of measure zero). We say that $\mathcal{B}$ has complexity at most $m$ if it is generated by at most $m$ sets. If $\mathcal{B}_X$ is a finite factor of $X$ with atoms $A_1,...,A_M$ and $\mathcal{B}_Y$ is a finite factor of $Y$ with atoms $B_1,...,B_N$ then $\mathcal{B}_X \vee \mathcal{B}_Y$ is a finite factor of $X \times Y$ with atoms $A_i \times B_j, 1\leq i \leq M, 1\leq j \leq N.$
\begin{thm}[Szemer\'edi's Regularity Lemma \cite{TA3}\footnote{This theorem is proved for counting measure in \cite{TA3} but the proof would work for any product measure on the product spaces.}] 
Let $f:\bX_{[d]} \rightarrow [0,1]$ be measurable, let $\tau >0$ and $F:\N \rightarrow \N$ be arbitrary increasing functions (possibly depends on $\tau$). Then there is an integer $M=O_{F,\tau}(1),$ factors $\mathcal{B}_I (I \subseteq [d],|I|=d-1)$ on $\bXi$ of complexity at most $M$ such that $f=f_1+f_2+f_3$ where
\begin{itemize}
	\item $f_1=\E({f| \bigvee_{I \subseteq [d],|I|=d-1} \mathcal{B}_I}).$
	\item $\left\| f_2 \right\|_{L^2_{\nu}} \leq \tau.$
	\item $\left\| f_3 \right\|_{\Box^d_{\nu}} \leq F(M)^{-1}.$
	\item $f_1,f_1+f_2 \in [0,1].$
\end{itemize}
\end{thm}
\begin{remark}
A consequence from this lemma that we will use later is the following: since $f_1$ is a constant on each atom of $\bigvee_{I.|I|=d-1} \mathcal{B}_I$, we can decompose $f_1$ as a finite sum of lower complexity functions i.e.a finite sum of product $\prod_{i=1}^d J_i$ where $J_i $ is a function in $\x_{[d] \backslash \{i\}}$ variable and takes values in $[0,1].$
\end{remark}
\begin{thm}[Weighted Simplex-Removal Lemma]
 Suppose $f^{(i)}(\x_{[d+1]\backslash \{i \}}) \leq \nu_{[d+1]\backslash \{i\}}(\x_{[d+1]\backslash \{i \}})$.
Let $\epsilon>0$, Suppose $|\Lambda| \leq \epsilon$ then there exist $E_i \subseteq \prod_{j \in [d+1] \backslash \{i\}}X_j$ such that for $1 \leq i \leq d+1$,
\begin{itemize}
\item $\displaystyle{\prod_{i \in [d+1]}}\1_{E_i} \equiv 0$
\item $\int_{X_1} \cdots \int_{X_{d+1}} f^{(i)}\1_{E_i^C}d \mu_{X_1} \cdots d\mu_{X_{d+1}}= \E_{\x_{[d+1] \backslash \{i\}}}\1_{E_i^C}f^{(i)}(\x_{[d+1]\backslash\{i\}})\prod_{J \subsetneq [d+1]\backslash\{i\}}\nu_J(\x_J) \leq \delta(\epsilon) $
\end{itemize}
where $\delta(\epsilon) \rightarrow 0$ as $\epsilon \rightarrow 0.$
\end{thm}

\begin{proof}
Using the transference principle (Theorem 4.6) for $1 \leq i \leq d+1$, write $f^{(i)}=g^{(i)}+h^{(i)}+k^{(i)}$ where
\begin{enumerate}
	\item $f^{(i)}=g^{(i)}+h^{(i)}+k^{(i)} $
	\item $0 \leq g^{(i)} \leq 2,~~~~~~~~~~~~~~ \text{supp}(g^{(i)}) \subseteq \Omega^{(i)}(T)$
	\item $\left\| h^{(i)} \right\|_{\Box^d_{\nu}} \leq \eta, ~~~~~~~~~~~~~~ \text{supp}(h^{(i)}) \subseteq \Omega^{(i)}(T)$
	\item  $k^{(i)}=f^{(i)} \cdot \1_{(\Omega^{(i)})^C(T)}$
\end{enumerate}
where
\[
\Omega^{(i)}(T)= \{ \x_{[d+1] \backslash \{i\}}: |\mathcal{D}f^{(i)}| \leq T \}, 1 \leq i \leq d
\]
\textbf{Step 1:} We'll show that if $T \geq T(\epsilon)$ is sufficiently large then
\[
\Lambda_{d+1}(g^{(1)}+h^{(1)},...,g^{(d+1)}+h^{(d+1)})=\Lambda_{d+1}(f^{(1)}-k^{(1)},...,f^{(d+1)}-k^{(d+1)}) \lesssim \epsilon.
\]
\textbf{\textsl{Proof of Step 1:}}
For $I \subseteq [d+1]$, the term on LHS can be written as a sum of the following terms:
\[
\Lambda_{d+1,I}(e^{(1)},...,e^{(d)},e^{(d+1)}), \quad e^{(i)} = \begin{cases} -k^{(i)} &\text{if}~~~~~~~~ i \in I \\f^{(i)} &\text{if}~~~~~~~~ i \notin I \end{cases}
\]
If $I = \emptyset$ then $\Lambda_{d+1}(f^{(1)},...,f^{(d)},f^{(d+1)}) \leq \epsilon$   ~~~~~~by the assumption. Suppose $I=\{i_1,...,i_r \} \neq \emptyset$ then
\begin{align*}
|\Lambda_{d+1,I}(e^{(1)},...,e^{(d)},f^{(d+1)})|
&=\bigg|\int_{X_1} \cdots \int_{X_{d+1}}f^{(1)} \cdots f^{(d+1)} \cdot \prod_{i \in I}\1_{(\Omega^{(i)})^C} d\mu_{X_1} \cdots d\mu_{X_{d+1}} \bigg|\\
&\leq \E_{\x_{[d+1]}}\prod_{I \subseteq [d+1],|I| \leq d}\nu_I(\x_I) \1_{(\Omega^{(i_1)})^C} \\
&\leq \frac 1T \E_{\x_{d+1}}\E_{\y_{[d+1] \backslash \{i_1\}}}\prod_{I \subseteq [d+1],|I| \leq d}\nu_I(\x_I) \prod_{\substack{\w_I \neq \underline{0}\\I \subseteq [d+1]\backslash\{i_1\}}}\nu_I(P_{\w_I}(\x_I,\y_I)) \\
&\lesssim \frac 1T
\end{align*}
by linear form condition.\\
\textbf{
Step 2}~~~~~~~~~~~~~~~~~~~~ We'll show $\Lambda_{d+1}(g^{(1)},...,g^{(d+1)}) \lesssim \epsilon$ \quad if $\eta \leq \eta(\epsilon),N \geq N(\epsilon,\eta).$\\
\textsl{Proof of step 2:} Write $g^{(i)}=g^{(i)}+h^{(i)}-h^{(i)}=f^{(i)} \cdot \1_{\Omega^{(i)}(T)}-h^{(i)}$ then we have $$0 \leq f^{(i)}\cdot \1_{\Omega^{(i)}(T)} \leq \nu_i, \| h^{(i)} \|_{\Box_{\nu}^d} \leq \eta$$ so by the weighted von-Neumann inequality and step 1 , we have
\begin{align*}
|\Lambda_{d+1}(g^{(1)},...,g^{(d+1)})| &= |\Lambda_{d+1}(g^{(1)}+h^{(1)},...,g^{(d+1)}+h^{(d+1)})-\Lambda_{d+1}(h^{(1)},..,h^{(d)},h^{(d+1)})| \\
&\lesssim \epsilon+\eta+o_{N \rightarrow \infty}(1)\\
&\lesssim \epsilon
\end{align*}
if $\tau \leq \epsilon, \eta \leq \epsilon, N \geq N(\epsilon)$  and the proof of step 2 is completed.\\
Now since $0 \leq g^{(i)}\leq 2 $ then (after normalizing) using the \textit{ordinary hypergraph removal lemma}(Theorem 5.1), we have
\[
F_{(i)} \subseteq \X_{[d+1] \backslash \{i\}}\quad \text{such that} \quad \prod_{1 \leq k \leq d+1}\1_{F_k} \equiv 0 \quad \text{and}
\]
\[
\int_{X_1} \cdots \int_{X_{d+1}} g^{(i)} \cdot \1_{F_i^C}~~~~~~~~~~~~~~~~~~~~ d\mu_{X_1} \cdots d\mu_{X_d+1} \lesssim \delta(\epsilon)
\]
so
\begin{align*}
\int_{X_1} \cdots \int_{X_{d+1}} f^{(i)} \cdot \1_{F_i^C}~~~~~~~~~~~~~~~~~~~~ d\mu_{X_1} \cdots d\mu_{X_{d+1}} &\lesssim \delta(\epsilon)+ \underbrace{\int_{X_1} \cdots \int_{X_{d+1}} h^{(i)} \cdot \1_{F_i^C}~~~~~~~~~~~~~~~~~~~~ d\mu_{X_1} \cdots d\mu_{X_d}d\mu_{X_{d+1}}}_{(A)}+ \\
&+\underbrace{\int_{X_1} \cdots \int_{X_{d+1}} f^{(i)} \cdot \1_{\Omega_i^C(T)}\1_{F_i^C}~~~~~~~~~~~~~~~~~~~~ d\mu_{X_1} \cdots d\mu_{X_{d+1}}}_{(B)}
\end{align*}
Now for our purpose, it suffices to show $(A),(B) \lesssim \epsilon.$ \\
\textsl{Estimate for (A):} By the \textit{ regularity lemma} \footnote{We need this since we don't have something like $\|fg \|_{\Box_{\nu}} \leq \|f \|_{\Box_{\nu}}\|g \|_{\Box_{\nu}}$}, the function $\1_{F_i^C}$ could be written as a sum of $O(1)$ of functions of the form
$\prod_{j \in [d+1]\backslash \{i\}}v_j^{(i)}$
plus some functions which give a small error term $O(\epsilon)$ in (A) (using von Neumann's inequality) where $v^{(i)}_j$ is a $[0,1]$- valued function  in $\x_{[d+1]\backslash\{i,j\}}$. We could write $u^{(i)}_j=\max v^{(i)}_j$ for each fixed $i,j$ then the sum of $\prod_{j \in [d+1]\backslash\{i\}}v_j^{(i)}$ is less than $C\prod_{j \in [d+1]\backslash\{i\}}u_j^{(i)}$ for some absolute constant $C$.
Applying Cauchy-Schwartz's inequality $d$ times to estimate the expression (A) (here let assume $i<d+1$, the case $i=d+1$ is the same.) :
\begin{align*}
& \bigg(\int_{X_1} \cdots \int_{X_d}\int_{X_{d+1}} h^{(i)} \cdot \1_{F_i^C}~~~~~~~~~~~~~~~~~~~~ d\mu_{X_1} \cdots d\mu_{X_d}d\mu_{X_{d+1}} \bigg)^{2^d} \\
&\lesssim \bigg[ \bigg( \int_{X_1} \cdots \int_{X_d} \bigg( \int_{X_{d+1}} h^{(i)} \prod_{\substack{1 \leq j \leq d \\ j \neq i}}u^{(i)}_j d\mu_{X_{d+1}} \bigg) u^{i}_{d+1} d\mu_{X_1} \cdots d\mu_{X_d}d\mu_{X_{d+1}}\bigg)^2 \bigg]^{2^{d-1}}\\
&\leq \bigg[\int_{X_1} \cdots \int_{X_d} \bigg( \int_{X_{d+1}} h^{(i)}\prod_{\substack{1 \leq j \leq d-1 \\ j \neq i}}u^{(i)}_j d\mu_{X_{d+1}}\bigg)^2d\mu_{X_1} \cdots d\mu_{X_d} \\
&\quad \times \int_{X_1} \cdots \int_{X_d} \big( u^{i}_{d+1} \big)^2 d\mu_{X_1} \cdots d\mu_{X_d}\bigg]^{2^{d-1}} \\
&\lesssim \bigg[\int_{X_1} \cdots \int_{X_d} \int_{X_{d+1}}\int_{Y_{d+1}}  h^{(i)}(\x_{[d+1] \backslash \{i\}},x_{d+1})h^{(i)}(\x_{[d] \backslash\{i\}},y_{d+1}) \\
&\prod_{\substack{1 \leq j \leq d \\ j \neq i}}u^{(i)}_j(\x_{[d] \backslash\{i\}},x_{d+1})u^{(i)}_j(\x_{[d] \backslash\{i\}},y_{d+1})  d\mu_{X_1} \cdots d\mu_{X_{d+1}}d\mu_{Y_{d+1}} \bigg]^{2^{d-1}}  \\
\end{align*}
Continue applying cauchy schwartz's inequality this way. After $d$ application of cauchy Schwartz's inequality, the $u^{(i)}_j$ eventually disappears and we have this bounded by $ \| h^{(i)}\|_{\Box_{\nu}}^{2^d} \leq \epsilon.$

\textsl{Estimate for (B)}: Next we estimate the expression in (B),
\begin{align*}
&\bigg|\int_{X_1} \cdots \int_{X_{d+1}} f^{(i)}\cdot \1_{(\Omega^{(i)}(T))^C} \cdot \1_{F_i^C} d\mu_{X_1} \cdots d\mu_{X_{d+1}}\bigg| \\
&\leq \int_{X_1} \cdots \int_{X_{d+1}} (\nu_{[d+1]\backslash\{i\}})\cdot \1_{(\Omega^{(i)}(T))^C} d\mu_{X_1} \cdots d\mu_{X_{d+1}} \\
&\leq \frac1T \E_{\x_{[d+1] \backslash \{i\}}}\E_{\y_{[d+1] \backslash\{i\}}}\nu_{[d+1] \backslash \{i\}}(\x_{[d+1] \backslash \{i\}})\prod_{|I|\leq d, i \notin I}\nu_I(\x_I)\prod_{\substack{\w_I \neq \underline{0}\\I \subseteq [d+1]\backslash\{i\}}}\nu_I(P_{\w_I}(\x_I,\y_I))\\
&\lesssim \frac 1T,
\end{align*}
by the linear forms condition. Hence if we choose sufficiently large $T$ then
\[
\int_{X_1} \cdots \int_{X_{d+1}} f^{(i)} \cdot \1_{F_i^C} d\mu_{X_1} \cdots d\mu_{X_{d+1}} \lesssim \delta(\epsilon).
\]
\end{proof}

\section{proof of the main result}
\subsection{From $\Z_N$ to $\Z$}
Now recall that $\nu_{\delta_1,\delta_2}(n) \approx \frac{\phi(W)}{W}\log N, \delta_1N \leq n \leq \de_2N, \de_1,\de_2 \in (0,1]$ for a sufficiently large prime $N$ in the residue class ~~~~~~~$b \pmod{W}$
By pigeonhole principle choose a $\underline{b} \in (\Z_W^{\times })^d$ such that
\[
|A \cap (W\Z)^d+b| \geq \alpha \frac{N^d}{(\log^d N) \phi(W)^d}
\]
Now consider $A_b=\{n \in [1,N/W]^d: Wn+b \in A \}$ and let $\delta_2 \in (0,1)$ then by the Prime Number Theorem there is a prime $N'$ such that $\de_2N'=(1+\delta)\frac{N}{W}$  for arbitrarily small real number $\delta$. Then if $N$ is sufficiently large and $\delta$ is sufficiently small with respect to $\alpha$ then
\eq \label{estimate}
|A_b \cap [1,\de_2N']^d| \geq \frac{\alpha \de_2^d}{2}\frac{(N'W)^d}{(\log^d N') \phi(W)^d}
\ee
On the other hand by Dirichlet's theorem on primes in arithmetic progressions, the number $n \in [1,N']^d \backslash [\de_1N',N']^d$ for which $Wn+b \in \mathbb{P}^d$ is $\leq c_d\epsilon_1\frac{(N'W)^d}{\log^d N' \phi(W)^d}$ Hence the estimate \eqref{estimate} holds for $A':=A_W \cap [\de_1N',\de_2N']^d$ as well provided that $\de_1$ is small enough. \par
Now we may consider $A'$ in place of $A$ (we are working in the group $\Z^d_{N'}$ instead). Now if we identify the group $\Z_{N'}^d$ with $[-\frac{N'}{2},\frac{N'}{2}]^d$ then for a sufficiently small $\de_1,\de_2$ any points in $A'$ are the same when we change from $\Z_{N'}^d$ to $[-\frac{N'}{2},\frac{N'}{2}]^d$ (no wrap around issue).

\subsection{Proof of the Main Theorem}
To prove the theorem, suppose on the contrary that  $A'$ contains less than $\epsilon \frac{N'^{d+1}}{(\log N')^{2d}}$ corners.$(\epsilon=c(\alpha))$ then
\begin{align*}
&\Lambda_{d+1}(f^{(1)},...,f^{(d+1)}) \\
&=(N')^{-(d+1)}\sum_{\x_{[d+1]}}\prod_{1 \leq i \leq d}\1_{A'}(x_1,...,x_{i-1},x_{d+1}-\sum_{\substack{1 \leq j \leq d \\ j \neq i}},x_{i+1},...,x_d)\nu_I \1_{A'}(x_1,...,x_d)\cdot \nu(x_1)...\nu(x_d)\\
 &\leq \frac{1}{N'^{d+1}}\sum_{\substack{p_i \in A',1\leq i \leq 2d \\ \text{that consitutes a corner}}}\prod_{1 \leq k \leq d}1_{A'}(p_1,...,p_{k-1},p_{d+k},p_{k+1},...,p_d)\1_A(p_1,...,p_d)\nu(p_1) ... \nu(p_{2d}) \\
&\lesssim \frac{1}{N'^{d+1}}\bigg(\frac{\phi(W)\log N'}{W} \bigg)^{2d}\times( \text{The number of corners in $A'$}) \\
&\leq \epsilon
\end{align*}

Now assume that $\Lambda_{d+1}(f^{(1)},...,f^{(d)},f^{(d+1)}) \lesssim \epsilon$ then by the relative hypergraph removal lemma
\[
\exists E_i,1 \leq i \leq d+1, E_i \subseteq \X_{[d+1]\backslash \{i\}} := \tilde{X}_i,
\]
such that
\[
\prod_{1 \leq i \leq d+1}\1_{E_i} \equiv 0, \quad \int_{\tilde{X}_i}f^{(i)}\1_{E_i^C}d\mu_{\tilde{X}_i} \lesssim \delta(\epsilon)
\]
where $\delta(\epsilon) \rightarrow 0$ as $\epsilon \rightarrow 0.$
Let $A'=A\cap [\delta_1N,\delta_2N]^d, z=\sum_{1 \leq j \leq d}x_j,g_{A'}:=g \cdot \1_{A'}$ for any function $g$ then
\begin{align*}
\tilde{\Lambda} &:=N'^{-d}\sum_{(x_1,...,x_d) \in A'}f_{A'}^{(1)}(x_2,...,x_d,z)f_{A'}^{(2)}(x_1,x_3,...,x_d,z)...f_{A'}^{(d)}(x_1,x_2,...,x_{d-1},z)f^{(d+1)}_{A'}(x_1,...,x_d) \\
&\geq N'^{-d}\sum_{(x_1,...,x_d) \in A'} \nu (x_1)...\nu(x_d) \\
&\gtrsim (N')^{-d}\big(\frac{\phi(W)}{W}\log N' \big)^d \cdot \frac{\alpha \cdot (N'W)^d}{(\phi(W)\log N')^d}
=\alpha.
\end{align*}
for arbitrarily large $N'$. Now
\begin{align*}
\tilde{\Lambda} &= \E_{\x_{[d]}}(f_{A'}^{(1)}\1_{E_1}+f_{A'}^{(1)}\1_{E_1^C})...(f_{A'}^{(d+1)}\1_{E_d+1}+f_{A'}^{(d+1)}\1_{E_{d+1}^C})
\end{align*}
Now we have by the assumption
$
\E_{\x_{[d]}}f_{A'}^{(1)} \cdot \1_{E_1}...f_{A'}^{(d+1)} \cdot \1_{E_{d+1}}  \equiv 0
$
so we just need to estimate each other term individually.\\
Consider $\E_{\x_{[d]}}f_{A'}^{(1)} \cdot \1_{E_1^C}f_{A'}^{(2)} \cdot \1_{E_2^{\pm}}...f_{A'}^{(d+1)} \cdot \1_{E_{d+1}^{\pm}}$, where $F^{\pm}$ can be either  $F$ or $F^C$ for any set $F$. Now since
\[
0 \leq f_{A'}^{(j)}\1_{E_j^{\pm}} \leq \nu(x_j), d \geq j \geq 2 \quad \text{and} \quad 0 \leq f_A^{(d+1)} \leq 1
\]
We have
\begin{align*}
\E_{\x_{[d]}}f_{A'}^{(1)} \cdot \1_{E_1^C}f_{A'}^{(2)} \cdot \1_{E_1^{\pm}}...f_{A'}^{(d+1)}\cdot \1_{E_{d+1}^{\pm}} &\leq \E_{\x_{[d]}}f_{A'}^{(1)} \cdot \1_{E_1^C}\nu(x_2)...\nu(x_d) \\
&= \int_{\tilde{X}_1} f^{(1)} \cdot \1_{E_1^C} d \mu_{X_2} \cdots d \mu_{X_{d+1}}  \\
&\lesssim  \delta(\epsilon).
\end{align*}
In the same way, we have for any $1 \leq i \leq d+1,$
\[
\E_{\x_{[d]}}f_{A'}^{(i)} \cdot \1_{E_i^C}\prod_{1 \leq j \leq d+1 , j \neq i}(f^{(j)}\cdot \1_{E_j^{\pm}}) \lesssim \delta(\epsilon)
\]
So if $N'>N(\alpha)$ then
\[
\E_{\x_{[d]}} f_{A'}^{(1)}(x_2,...,x_d,u)f_{A'}^{(2)}(x_1,x_3,...,x_d,u)...f_{A'}^{(d)}(x_1,...,x_{d-1},u)f^{(d+1)}_{A'}(x_1,...,x_d) \lesssim \delta(\epsilon) = o(\alpha)
\]
This is a contradiction. Hence there are $\gtrsim \epsilon\frac{N'^{d+1}}{(\log N')^{2d}}$ corners in $A$. Note that the number of degenerated corners is at most $O(\frac{N'^d}{(\log N')^d})$ as the corner is degenerated (and will be degenerated into a single point ) iff $z=\sum_{1 \leq j \leq d}x_j.$ \\

\appendix
\section{The Green-Tao Measure and pseudorandomness}
\subsection{Pseudorandom Measure Majorizing Primes}
Let us recall the \textit{Mangoldt function} which in many problems is used to replace the indicator function of the set of primes.
\[
 \Lambda(n) =\begin{cases}  \log p &\text{if}~~~~~~~~~~~~~~~~~~~~~~~~~~~~~~~~~~~~~n=p^k,k \geq 1 \\
 0  &\text{otherwise} \end{cases}
\]
Primes has \textit{local obstructions} that prevents them from being truly random; $\Lambda(n)$ is concentrated on just $\phi(q)$ residue classes $\pmod{q}.$ To get rid of this kind of obstruction on all small residue classes Green and Tao introduced a device, the so-called \textbf{W-Trick} \cite{GT1} we we recall here. Let $\omega(N)$ be a sufficiently slowly growing function of $N$ and let $W=\prod_{p \leq \omega(N)}p$. If $b$ is any positive integer with $(b,W)=1$, then by the Prime Number Theorem, we have $W=\exp((1+o(1))\omega(N))$ and we have that $\mathbb{P}_{W,b}$ is uniformly distributed $\pmod{q}$ for $q \leq \omega(N)$. \par
Let $\mathbb{P}_{N,W,b}:=\{ n: Wn+b \in \mathbb{P_N} \}$ and define the modified von-Mangoldt function by

\begin{definition} For any fixed $(b,W)=1$, let
\label{Mangoldlt}
\begin{equation*}
\overline{\Lambda}_{b}(n)=
\begin{cases}
\frac{\phi(W)}{W}\log(Wn+b)&\text{if $Wn+b$ is prime.}\\
0&\text{otherwise,}
\end{cases}
\end{equation*}
moreover define
\[
\aLL_b^d(x_1,...,x_d)=\aLL_{b_1}(x_1) \cdots  \aLL_{b_d}(x_d) , b=(b_1,...,b_d) \in \Z_N^{\times d}
\]
\end{definition}
\noindent
Note that by the Prime Number Theorem in arithmetic progressions, we have $\E_{n \leq N}\overline{\Lambda}_{b}(n) \sim 1 .$
Let us recall now the definition of Green-Tao measure introduced in \cite{GT1}.
\begin{definition}[Goldston-Yildirim sum] \cite{GY1},\cite{GT1}
\[
\Lambda_R(n)= \sum_{d|n, d \leq R}\mu(d) \log \frac Rd
\]
\end{definition}
We may take $R=N^{d^{-1}2^{-d-5}}$
\begin{definition} [Green-Tao measure]
\label{Measure}
For given small parameters $1 \geq \delta_1,\delta_2>0$ , define a function $\nu_{\de_1,\de_2}:\Z_N \rightarrow \R$
\[
\nu_{\de_1,\de_2}(n)=\nu(n)=
\begin{cases}
\frac{\phi(W)}{W}\frac{\Lambda_R(Wn+b)^2}{\log R}&\text{if $\delta_1N \leq n \leq \delta_2N$}\\
0 &\text{otherwise}
\end{cases}
\]
\end{definition}

It is immediate from the definition that $\nu(n) \geq d^{-1}2^{-d-6} \overline{\Lambda}_b(n)$. Finally lets us recall the pseudo-randomness properties we used here; summarized in the following definitions.

\begin{definition}[Linear forms condition.] Let $m_0, t_0 \in \N$ be parameters then we say that $\nu$ satisfies  $(m_0,t_0)-$ linear form condition if for any $m \leq m_0, t \leq t_0$, suppose $\{a_{ij}\}_{\substack{1 \leq i \leq m \\ 1 \leq j \leq t}}$ are subsets of integers and $b_i \in \Z_N$. Given $m$ (affine) linear forms $L_i:\Z_N^t \rightarrow \Z_N$ with $L_i(x)=\sum_{1\leq j \leq t}a_{ij}x_j+b_i$ for $1 \leq i \leq m$ be such that each $\phi_i$ is nonzero and they are pairwise linearly independent over rational. Then
\[
\E (\prod_{1 \leq i \leq m} \nu(L_i(x)): x \in \Z_N^t)=1+o_{N \rightarrow \infty,m_0,t_0}(1)
\]
\end{definition}

\begin{definition}[Correlation condition.]
We say that a measure $\nu$ satisfies $(m_0,m_1,...,m_{l_2})-$ correlation condition if there is a function $\tau:\Z_N \rightarrow \R_{+}$ such that
\begin{enumerate}
	\item $\E(\tau(x)^m:x \in \Z_N)=O_m(1)$ for any $m \in \Z_+$
	\item Suppose
	
\begin{itemize}
	\item $\phi_i,\psi^{(k)}:\Z_N^t \rightarrow \Z_N (1 \leq i \leq l_1,1 \leq k \leq l_2, l_1+l_2 \leq m_0)$ are all pairwise linearly independent (over $\Z$)  linear forms
	\item For each $1 \leq g \leq l_2, 1 \leq j < j' \leq m_g$ we have $a_j^g \neq 0,$  and $a_j^{(g)}\psi^{(g)}(x)+h_j^{(g)}, a_{j'}^{(g)}\psi^{(g)}(x)+h_{j'}^{(g)}$ are different (affine) linear forms.
	\end{itemize}
then, we have
\begin{displaymath}
 \E_{x \in Z_N^d}\prod_{k=1}^{l_1}\nu(\phi_k(x))\prod_{k=1}^{l_2}\prod_{j=1}^{m_k}\nu(a_j^{(k)}\psi^{(k)}(x)+h_j^{(k)}) \leq \prod_{k=1}^{l_2}\sum_{1\leq j<j' \leq m_k}\tau \bigg(W(a^{(k)}_{j'}h^{(k)}_{j}-a^{(k)}_{j}h^{(k)}_{j'})+(a_{j'}^{(k)}-a_j^{(k)})b \bigg)
\end{displaymath}	
\end{enumerate}
where $W= \prod_{p \leq \omega(N)}p$.
\end{definition}

\begin{thm}
The green-Tao measure $\nu$ satisfies linear forms and correlation conditions on any parameters that may depend on $d$ or $\alpha$ (not in $N$).
\end{thm}

The proof of the linear forms condition is given in \cite{GT1}, as well as the proof of a slightly simpler form of the correlation condition. The above correlation condition is essentially the same as the one given in \cite{MC1}, Proposition 4. In fact the argument there works without any modification, save for a minor change in calculating the so-called local factors, see Lemma 3 there. We omit the details.

\end{document}